\documentclass[12pt]{article}
\usepackage{graphicx}
\usepackage{amsmath}
\usepackage{amsthm}
\usepackage{amssymb}
\usepackage{enumerate}
\usepackage{subeqnarray}
\def\a{\alpha}
\def\b{\beta}
\def\d{\delta}
\def\g{\gamma}
\def\G{\Gamma}

\def\s{\sigma}
\def\S{\Sigma}

\def\z{\zeta}
\def\ve{\varepsilon}
\def\wt{\widetilde}
\def\bb{\mathbb}

\numberwithin{equation}{section}

\newtheorem{rem}{\qquad REMARK}
\newtheorem{lem}{\qquad LEMMA}
\newtheorem{thm}{\qquad THEOREM}
\newtheorem{prop}{\qquad PROPOSITION}

\DeclareMathOperator*\Res{Res}
\DeclareMathOperator\Ai{Ai}
\DeclareMathOperator\Bi{Bi}
\DeclareMathOperator\re{Re}
\DeclareMathOperator\im{Im}
\DeclareMathOperator\OR{\ or\ }
\DeclareMathOperator\AND{\ and\ }

\begin{document}
\title{Uniform Asymptotics of the Meixner Polynomials}

\author{by \\ \\ X. S. Wang${}^{1,2,3}$ and R. Wong${}^2$}
\date{}

\maketitle

\begin{abstract}
Using the steepest descent method of Deift-Zhou,
we derive locally uniform asymptotic formulas for the Meixner polynomials.
These include an asymptotic formula in a neighborhood of the origin,
a result which as far as we are aware has not yet been obtained previously.
This particular formula involves a special function,
which is the uniformly bounded solution to a scalar Riemann-Hilbert problem,
and which is asymptotically (as the polynomial degree $n$ tends to infinity) equal to the constant $``1"$
except at the origin.
Numerical computation by using our formulas, and comparison with earlier results, are also given.
\end{abstract}
\vfill

\hrule width 6cm
\bigskip

\ \ \ \hspace*{4mm}

\begin{enumerate}[${}^1$]

\item Department of Mathematics, University of Science and
Technology of China, Hefei, Anhui, China.

\item Department of Mathematics, City University of Hong Kong, Tat
Chee Avenue, Kowloon, Hong Kong.

\item Joint Advanced Research Center, University of Science and
Technology of China-City University
  of Hong Kong, Suzhou, Jiangshu,
   China.

\end{enumerate}

\newpage

\section{Introduction}

For $\b>0$ and $0<c<1$, the Meixner polynomials are explicitly given by
\cite[(1.9.1)]{KS98}
\begin{eqnarray}\label{Meixner}
  M_n(z;\b,c)={}_2F_1\bigg(\begin{array}{c}
    -n,-z\\ \b
  \end{array}
  \bigg|1-\frac 1 c\bigg)=\sum_{k=0}^n\frac{(-n)_k(-z)_k}{(\b)_kk!}(1-\frac 1 c)^k.
\end{eqnarray}
They satisfy the discrete orthogonality condition \cite[(1.9.2)]{KS98}
\begin{eqnarray}\label{orthogonality}
  \sum_{k=0}^\infty\frac{(\b)_k}{k!}c^k M_m(k;\b,c)M_n(k;\b,c)=\frac{c^{-n}n!}{(\b)_n(1-c)^\b}\d_{mn},
\end{eqnarray}
and the recurrence relation \cite[(1.9.3)]{KS98}
\begin{eqnarray}\label{recurrence relation}
  zM_n(z)=\frac{c(n+\b)}{c-1}M_{n+1}(z)-\frac{n+(n+\b)c}{c-1}M_n(z)+\frac n{c-1}M_{n-1}(z).
\end{eqnarray}

Not much is known about the asymptotic behavior of the Meixner polynomials for large values of $n$.
Using probabilistic arguments, Maejima and Van Assche \cite{MV85} have given an asymptotic formula
for $M_n(n\a;\b,c)$ when $\a<0$ and $\b$ is a positive integer.
Their result is in terms of elementary functions.
In \cite{JW98}, Jin and Wong have used the steepest-descent method for integrals
to derive two infinite asymptotic expansions for $M_n(n\a;\b,c)$. 
One holds uniformly for $0<\ve\le \a\le 1+\ve$, and the other holds uniformly for $1-\ve\le \a\le M<\infty$;
both expansions involve the parabolic cylinder function and its derivative.

In view of Gauss's contiguous relations for hypergeometric functions
(\cite[Section 15.2]{AS70}), we may restrict our study to the case $1\le\b<2$.
Fixing any $0<c<1$ and $1\le\b<2$, we intend to investigate the large-$n$ behavior of
$M_n(nz-\b/2;\b,c)$ for $z$ in the whole complex plane.
Our approach is based on the steepest descent method for oscillatory Riemann-Hilbert problems,
first introduced by P. Deift and X. Zhou \cite{DZ93} for nonlinear partial differential equations,
and later developed in \cite{DKMVZ99} and \cite{BKMM03,BKMM07} for orthogonal polynomials
with respect to exponential weights or a general class of discrete weights.

The material in this paper is arranged as follows.
In Section 2, we study the basic interpolation problem whose solution $P(z)$ can be solved
explicitly in terms of the Meixner polynomials and their Cauchy tranformations.
In Section 3, we make the first transformation $P\to Q$ which includes a rescaling.
The second transformation $Q\to R$ is introduced in Section 4,
which removes the poles in the interpolation problem.
As a consequence, we have created several jump discontinuities across certain contours in the complex plane.
In Section 5, we derive the equilibrium measure corresponding to the Meixner polynomials.
This measure is used in the third transformation $R\to S$ introduced in Section 6.
In Section 7, we give the final transformation $S\to T$ in connection with the factorization of the jump matrices and the deformation of the contours.
In Section 8, we construct the parametrix $T_{par}(z)$ which is an approximate solution to the Riemann-Hilbert problem for $T(z)$.
Our main theorem is stated in Section 9.
In Section 10, we provide some numerical evidence to demonstrate the accuracy of our results.
In Section 11, we compare our formulas with those which already exist in the literature.

\section{The Basic Interpolation Problem}

From (\ref{Meixner}), the monic Meixner polynomials are given by
\begin{eqnarray}\label{monic Meixner}
  \pi_n(z):=(\b)_n(1-\frac 1 c)^{-n}M_n(z;\b,c).
\end{eqnarray}
On account of (\ref{recurrence relation}), we obtain the recurrence relation
\begin{eqnarray}\label{monic recurrence relation}
  z\pi_n(z)=\pi_{n+1}(z)+\frac{n+(n+\b)c}{1-c}\pi_n(z)+\frac {n(n+\b-1)c}{(1-c)^2}\pi_{n-1}(z).
\end{eqnarray}
The orthogonality property of $\pi_n(z)$ can be derived from (\ref{orthogonality}), and we have
\begin{eqnarray}\label{monic orthogonality}
  \sum_{k=0}^\infty \pi_m(k)\pi_n(k)w(k)=\d_{mn}/\g_n^2,
\end{eqnarray}
where
\begin{eqnarray}\label{gamma}
  \g_n^2=\frac{(1-c)^{2n+\b}c^{-n}}{\G(n+\b)\G(n+1)}
\end{eqnarray}
and
\begin{eqnarray}\label{weight}
  w(z):=\frac{\G(z+\b)}{\G(z+1)}c^z.
\end{eqnarray}
Let $P(z)$ be the $2\times2$ matrix defined by
\begin{eqnarray}\label{P}
  P(z):=\left(\begin{matrix}
    \pi_n(z)&\sum\limits_{k=0}^\infty\cfrac{\pi_n(k)w(k)}{z-k}\\
    \\
    \g_{n-1}^2\pi_{n-1}(z)&\sum\limits_{k=0}^\infty\cfrac{\g_{n-1}^2\pi_{n-1}(k)w(k)}{z-k}
  \end{matrix}\right).
\end{eqnarray}
  For consistency, we shall use capital letters to denote matrix-valued functions that depend on the large parameter $n$.
  Therefore, all the matrices $P, Q, R, S, T, M$ and $K$ depend on both $z$ and $n$.
The following proposition states that $P(z)$ is the unique solution to an interpolation problem,
which is the discrete analogue of the Riemann-Hilbert problem corresponding to
the orthogonal polynomials with continuous weights; see \cite{FIK91,FIK92}.
\begin{prop}\label{prop-P}
The matrix-valued function $P(z)$ defined in (\ref{P}) is the unique solution of the
following interpolation problem:
\begin{enumerate}[(P1)]
\item  $P(z)$ is analytic in
  $\bb{C}\setminus\bb{N}$;
\item at each $z=k\in\bb{N}$, the first
column of
  $P(z)$ is analytic and the second column of $P(z)$ has a
  simple pole with residue
  \begin{eqnarray}\label{P2}
  \Res_{z=k}P(z)=\lim_{z\to k}P(z)\left(\begin{matrix}
    0&w(z)\\
    0&0
  \end{matrix}\right)=\left(\begin{matrix}
    0\ &w(k)P_{11}(k)\\
    \\
    0&w(k)P_{21}(k)
  \end{matrix}\right);
  \end{eqnarray}
\item for $z$ bounded away from $\bb{N}$,  $P(z)\left(\begin{matrix}
    z^{-n}&0\\
    0&z^n
  \end{matrix}\right)=I+O(|z|^{-1})$ as $z\to\infty$.
\end{enumerate}
\end{prop}

\begin{proof}
  Since $w(k)$ decays exponentially to zero as $k\to+\infty$,
  the summations in the second column of $P(z)$ in (\ref{P})
  are uniformly convergent for $z$ in any compact subset of $\bb{C}\setminus\bb{N}$.
  Therefore, (P1) is obvious.

  For each $k\in\bb{N}$, we have from (\ref{P}) $$\Res\limits_{z=k}P_{12}(z)=\pi_n(k)w(k)=P_{11}(k)w(k),$$
  and $$\Res\limits_{z=k}P_{22}(z)=\g_{n-1}^2\pi_{n-1}(k)w(k)=P_{21}(k)w(k).$$
  Thus, (P2) follows.

  To prove (P3) we only need to show that $P_{12}(z)z^n=O(|z|^{-1})$
  and $P_{22}(z)z^n=1+O(|z|^{-1})$ as $z\to\infty$ and for $z$ bounded away from $\bb{N}$.
  Using the following expansion
  $$\cfrac{1}{z-k}=\sum\limits_{i=0}^{n-1}\cfrac{k^i}{z^{i+1}}+\cfrac{1}{z^{n+1}}\cfrac{k^n}{1-k/z},$$
  we have $$P_{12}(z)z^n=\sum\limits_{i=0}^{n-1}z^{n-i-1}\sum\limits_{k=0}^\infty k^i\pi_n(k)w(k)
  +\cfrac 1 z\sum\limits_{k=0}^\infty \cfrac{k^n\pi_n(k)w(k)}{1-k/z}.$$
  The orthogonality property (\ref{monic orthogonality}) implies that
  $\sum\limits_{k=0}^\infty k^i\pi_n(k)w(k)=0$ for any $i=0,1,\cdots,n-1$. Thus, we obtain
  $$P_{12}(z)z^n=\cfrac 1 z\sum\limits_{k=0}^\infty \cfrac{k^n\pi_n(k)w(k)}{1-k/z}.$$
  Since $z$ is bounded away from $\bb{N}$, it is easily seen that
  the last sum is uniformly bounded.
  Hence, we have $P_{12}(z)z^n=O(|z|^{-1})$ as $z\to\infty$.
  On the other hand, we also have
  $$P_{22}(z)z^n=\sum\limits_{i=0}^{n-1}z^{n-i-1}\g_{n-1}^2\sum\limits_{k=0}^\infty k^i\pi_{n-1}(k)w(k)
  +\cfrac 1 z\sum\limits_{k=0}^\infty \cfrac{\g_{n-1}^2k^n\pi_{n-1}(k)w(k)}{1-k/z}.$$
  Again, using the orthogonality property (\ref{monic orthogonality}), we obtain
  $\sum\limits_{k=0}^\infty k^i\pi_{n-1}(k)w(k)=\d_{i,n-1}/\g_{n-1}^2$ for any $i=0,1,\cdots,n-1$.
  Thus, it is readily seen that $P_{22}(z)z^n=1+O(|z|^{-1})$ as $z\to\infty$ and for $z$ bounded away from $\bb{N}$.
  This ends our proof of (P3).

  The uniqueness of the solution follows from Liouville's theorem.
  Indeed, condition (\ref{P2}) implies that
  the residue of $\det P(z)$ at $k\in\bb{N}$ is zero.
  Thus, the determinant function $\det P(z)$ can be analytically continued to an entire function.
  Condition (P3), together with Liouville's theorem, implies that $\det P(z)=1$.
  Therefore, $P(z)$ is invertible in $\bb{C}\setminus\bb{N}$.
  Let $\wt P(z)$ be a second solution to the interpolation problem (P1)-(P3).
  It is easily seen that the residue of $P(z)\wt P^{-1}(z)$ at $k\in\bb{N}$ is zero.
  Hence, $P(z)\wt P^{-1}(z)$ can be extended to an entire function.
  Again, using condition (P3), we obtain from Liouville's theorem that $P(z)\wt P^{-1}(z)=I$.
  This establishes the uniqueness.
\end{proof}

\section{The First Transformation $P\to Q$}

The first transformation involves the following rescaling:
\begin{eqnarray}\label{U}
  U(z):=n^{-n\s_3}P(nz-\b/2)=\left(\begin{matrix}
    n^{-n}&0\\
    0&n^n
  \end{matrix}\right)P(nz-\b/2).
\end{eqnarray}
Here, $\s_3:=\left(\begin{matrix}
    1&0\\
    0&-1
  \end{matrix}\right)$ is a Pauli matrix.
In this paper, we will also make use of another Pauli matrix
$\s_1:=\left(\begin{matrix}
    0&1\\
    1&0
  \end{matrix}\right)$; see (\ref{Tpa}).
Let $\bb{X}$ denote the set defined by
\begin{eqnarray}\label{X}
  \bb{X}:=\{X_k\}_{k=0}^\infty,\indent \mbox{where}\ X_k:=\frac{k+\b/2}n.
\end{eqnarray}
The $X_k$'s are called {\it nodes}. Our first transformation is given by
\begin{eqnarray}\label{Q}
  Q(z)&:=&U(z)\bigg[\prod_{j=0}^{n-1}(z-X_j)\bigg]^{-\s_3}
  =n^{-n\s_3}P(nz-\b/2)\bigg[\prod_{j=0}^{n-1}(z-X_j)\bigg]^{-\s_3}\nonumber
  \\&=&\left(\begin{matrix}
    n^{-n}&0\\
    0&n^n
  \end{matrix}\right)P(nz-\b/2)\left(\begin{matrix}
    \prod\limits_{j=0}^{n-1}(z-X_j)^{-1}&0\\
    0&\prod\limits_{j=0}^{n-1}(z-X_j)
  \end{matrix}\right),
\end{eqnarray}
and the interpolation problem corresponding to $Q(z)$ is given below.
\begin{prop}\label{prop-Q}
The matrix-valued function $Q(z)$ defined in (\ref{Q}) is the unique solution of the
following interpolation problem:
\begin{enumerate}[(Q1)]
\item $Q(z)$ is analytic in
  $\mathbb{C}\setminus X$;
\item at each node $X_k$ with $k\in\mathbb{N}$ and $k\ge n$, the
  first column of
  $Q(z)$ is analytic and the second column of $Q(z)$ has a
  simple pole with residue
  \begin{eqnarray}\label{Q2a}
  \Res_{z=X_k}Q(z)=\lim_{z\to X_k}Q(z)\left(\begin{matrix}
    0&w(nz-\b/2)\prod\limits_{j=0}^{n-1}(z-X_j)^2\\
    0&0
  \end{matrix}\right);
  \end{eqnarray}
  at each node $X_k$ with $k\in\mathbb{N}$ and $k< n$, the
  second column of
  $Q(z)$ is analytic and the first column of $Q(z)$ has a
  simple pole with residue
  \begin{eqnarray}\label{Q2b}
  \Res_{z=X_k}Q(z)=\lim_{z\to X_k}Q(z)\left(\begin{matrix}
    0&0\\
    \cfrac{1}{w(nz-\b/2)}\prod\limits_{\substack{j=0\\j\ne k}}^{n-1}(z-X_j)^{-2}&0
  \end{matrix}\right);
  \end{eqnarray}
\item for $z$ bounded away from $\bb{X}$, $Q(z)=I+O(|z|^{-1})$ as $z\to\infty$.
\end{enumerate}
\end{prop}

\begin{proof}
  On account of (\ref{Q}),
  (Q1) and (Q3) follow from (P1) and (P3), respectively.

  Also from (\ref{Q}), we have
  $$Q_{11}(z)=n^{-n}P_{11}(nz-\b/2)\prod\limits_{j=0}^{n-1}(z-X_j)^{-1}$$
  and
  $$Q_{12}(z)=n^{-n}P_{12}(nz-\b/2)\prod\limits_{j=0}^{n-1}(z-X_j).$$
  At each node $z=X_k$ with $k\in\bb{N}$ and $k\ge n$,
  it is easily seen from (P2) that
  $Q_{11}(z)$ is analytic and $Q_{12}(z)$ has a simple pole,
  where the residue can be calculated as follows:
  \begin{eqnarray*}
  \Res_{z=X_k}Q_{12}(z)&=&n^{-n}\Res_{z=X_k}P_{12}(nz-\b/2)\prod\limits_{j=0}^{n-1}(X_k-X_j)\\
  &=&n^{-n}w(nX_k-\b/2)P_{11}(nX_k-\b/2)\prod\limits_{j=0}^{n-1}(X_k-X_j)\\
  &=&Q_{11}(X_k)w(nX_k-\b/2)\prod\limits_{j=0}^{n-1}(X_k-X_j)^2.
  \end{eqnarray*}
  Similarly, one can show from (\ref{Q}) and (P2) that
  $Q_{21}(z)$ is analytic and $Q_{22}(z)$ has a simple pole at $X_k$, $k\ge n$,
  with residue
  $$\Res\limits_{z=X_k}Q_{22}(z)=Q_{21}(X_k)w(nX_k-\b/2)\prod\limits_{j=0}^{n-1}(X_k-X_j)^2.$$
  This proves the first half of (Q2).

  Now, we compute the singularities of $Q(z)$ at the nodes $X_k$ with $k\in\bb{N}$ and $k<n$.
  First, it is easily seen from (\ref{Q}) and (P2) that
  $Q_{12}(z)$ can be analytically continued to the node $z=X_k$ and
  \begin{eqnarray*}
  \lim_{z\to X_k}Q_{12}(z)&=&\lim_{z\to X_k}n^{-n}P_{12}(nz-\b/2)\prod\limits_{j=0}^{n-1}(z-X_j)\\
  &=&n^{-n}\Res_{z=X_k}P_{12}(nz-\b/2)\prod\limits_{\substack{j=0\\j\ne k}}^{n-1}(X_k-X_j)\\
  &=&n^{-n}w(nX_k-\b/2)P_{11}(nX_k-\b/2)\prod\limits_{\substack{j=0\\j\ne k}}^{n-1}(X_k-X_j).
  \end{eqnarray*}
  Furthermore, since $P_{11}(nz-\b/2)$ is analytic at $z=X_k$ by (P2),
  the function $Q_{11}(z)$ has a simple pole at $z=X_k$ and from the last equation we obtain
  \begin{eqnarray*}
  \Res_{z=X_k}Q_{11}(z)&=&n^{-n}P_{11}(nX_k-\b/2)\prod\limits_{\substack{j=0\\j\ne k}}^{n-1}(X_k-X_j)^{-1}\\
  &=&Q_{12}(X_k)w(nX_k-\b/2)^{-1}\prod\limits_{\substack{j=0\\j\ne k}}^{n-1}(X_k-X_j)^{-2}.
  \end{eqnarray*}
  Similarly, we see from (\ref{Q}) and (P2) that
  $Q_{22}(z)$ is analytic and $Q_{21}(z)$ has a simple pole
  with residue
  \begin{eqnarray*}
  \Res\limits_{z=X_k}Q_{21}(z)&=&n^{-n}P_{21}(nX_k-\b/2)\prod\limits_{\substack{j=0\\j\ne k}}^{n-1}(X_k-X_j)^{-1}\\
  &=&Q_{22}(X_k)w(nX_k-\b/2)^{-1}\prod\limits_{\substack{j=0\\j\ne k}}^{n-1}(X_k-X_j)^{-2}.
  \end{eqnarray*}
  This proves the second half of (Q2).

  As in the proof of Proposition \ref{prop-P}, the uniqueness again follows from Liouville's theorem.
\end{proof}

\section{The Second Transformation $Q\to R$}
The purpose of the second transformation is to remove the poles in the interpolation problem
for $Q(z)$. For any fixed $0<c<1$ and $1\le\b<2$,
let $\d_0>0$ be a small number that will be determined in Remark \ref{rem-d}.
Fix any $0<\d<\d_0$, and define (cf. Figure \ref{fig-Q2R})
\begin{subequations}\label{R}
  \begin{equation}
    R(z):=Q(z)\left(\begin{matrix}
    1&0\\
    a_{21}^{(\pm)}&1
  \end{matrix}\right)
  \end{equation}
  for $\re z\in(0,1)\AND \im z\in(0,\pm\d)$, and
  \begin{equation}
    R(z):=Q(z)\left(\begin{matrix}
    1&a_{12}^{(\pm)}\\
    0&1
  \end{matrix}\right)
  \end{equation}
  for $\re z\in(0,1)\AND \im z\in(0,\pm\d)$, and
  \begin{equation}
    R(z):=Q(z)
  \end{equation}
  for $\re z\notin[0,\infty)\OR \im z\notin[-\d,\d]$,
\end{subequations}
where
$$a_{12}^{(\pm)}:=-\cfrac{n\pi w(nz-\b/2)\prod\limits_{j=0}^{n-1}(z-X_j)^2}{e^{\mp i\pi (nz-\b/2)}\sin(n\pi z-\b\pi/2)},\indent
a_{21}^{(\pm)}:=-\cfrac{e^{\pm i\pi(nz-\b/2)}\sin(n\pi z-\b\pi/2)}{n\pi w(nz-\b/2)\prod\limits_{j=0}^{n-1}(z-X_j)^2}.$$
\begin{figure}[htp]
\centering
\includegraphics{Q2R.eps}
\caption{The transformation $Q\to R$ and the contour $\S_R$.}\label{fig-Q2R}
\end{figure}
\begin{lem}\label{lem-R}
  For each $k\in\bb{N}$, the singularity of $R(z)$ at the node $X_k=\frac{k+\b/2}{n}$ is removable, that is,
  $\Res\limits_{z=X_k}R(z)=0.$
\end{lem}
\begin{proof}
  For any $k\in\bb{N}$ with $k\ge n$, we have $X_k=\frac{k+\b/2}{n}>1$ since $1\le\b<2$.
  For any complex $z$ with $\re z\in(1,\infty)$ and $\im z\in(0,\pm\d)$, we obtain from (\ref{R}) that $R_{11}(z)=Q_{11}(z)$ and
  \begin{eqnarray}\label{lem-R-11}
    R_{12}(z)=Q_{12}(z)-Q_{11}(z)
    \cfrac{n\pi w(nz-\b/2)\prod\limits_{j=0}^{n-1}(z-X_j)^2}{e^{\mp i\pi (nz-\b/2)}\sin(n\pi z-\b\pi/2)}.
  \end{eqnarray}
  The analyticity of the function $Q_{11}(z)$ at the node $X_k$ is clear from (Q2) in Proposition \ref{prop-Q}.
  Hence, the function $R_{11}(z)$ is analytic.
  To show that the singularity of the function $R_{12}(z)$ at the node $X_k$ is removable,
  we first note from (Q2) that
  \begin{eqnarray}\label{lem-R-12}
    \Res_{z=X_k}Q_{12}(z)=Q_{11}(X_k)w(nX_k-\b/2)\prod\limits_{j=0}^{n-1}(X_k-X_j)^2.
  \end{eqnarray}
  Furthermore, a direct calculation gives
  \begin{eqnarray}\label{lem-R-13}
    \Res_{z=X_k}\cfrac{n\pi w(nz-\b/2)\prod\limits_{j=0}^{n-1}(z-X_j)^2}{e^{\mp i\pi (nz-\b/2)}\sin(n\pi z-\b\pi/2)}
    =w(nX_k-\b/2)\prod\limits_{j=0}^{n-1}(X_k-X_j)^2.
  \end{eqnarray}
  Applying (\ref{lem-R-12}) and (\ref{lem-R-13}) to (\ref{lem-R-11}) yields $\Res\limits_{z=X_k}R_{12}(z)=0$.
  Similarly, we can prove that the functions $R_{21}(z)$ and $R_{22}(z)$ are analytic at the node $X_k$.

  Now, we consider the case $k\in\bb{N}$ with $k<n$.
  Since $1\le\b<2$, we have $X_k=\frac{k+\b/2}{n}<1$.
  For any $z$ with $\re z\in(0,1)$ and $\im z\in(0,\pm\d)$, we obtain from (\ref{R}) that $R_{12}(z)=Q_{12}(z)$ and
  \begin{eqnarray}\label{R-l-21}
    R_{11}(z)=Q_{11}(z)-Q_{12}(z)
    \cfrac{e^{\pm i\pi(nz-\b/2)}\sin(n\pi z-\b\pi/2)}{n\pi w(nz-\b/2)\prod\limits_{j=0}^{n-1}(z-X_j)^2}.
  \end{eqnarray}
  From (Q2) in Proposition \ref{prop-Q} we see that the function $Q_{12}(z)$,
  and hence the function $R_{12}(z)$, is analytic at the node $X_k$.
  Moreover, we have
  \begin{eqnarray}\label{R-l-22}
    \Res_{z=X_k}Q_{11}(z)=Q_{12}(X_k)\cfrac 1{w(nX_k-\b/2)}\prod\limits_{\substack{j=0\\j\ne k}}^{n-1}(X_k-X_j)^{-2}.
  \end{eqnarray}
  Since
  \begin{eqnarray}\label{R-l-23}
    \Res_{z=X_k}\cfrac{e^{\pm i\pi(nz-\b/2)}\sin(n\pi z-\b\pi/2)}{n\pi w(nz-\b/2)\prod\limits_{j=0}^{n-1}(z-X_j)^2}
    =\cfrac 1{w(nX_k-\b/2)}\prod\limits_{\substack{j=0\\j\ne k}}^{n-1}(X_k-X_j)^{-2},
  \end{eqnarray}
  we obtain from (\ref{R-l-21}) that $\Res\limits_{z=X_k}R_{11}(z)=0$.
  The analyticity of the second row in $R(z)$ at the node $X_k$ can be verified similarly.
  This ends our proof.
\end{proof}

From the definition of $R(z)$ in (\ref{R}) and the analyticity condition (Q1) of $Q(z)$ in Proposition \ref{prop-Q},
it is easily seen that $R(z)$ is analytic in $\mathbb{C}\setminus \S_R$,
where $\S_R$ is the oriented contour shown in Figure \ref{fig-Q2R}.
Denote by $R_+(z)$ the limiting value taken by $R(z)$ on $\S_R$ from the left and by $R_-(z)$
taken from the right.
We intend to calculate the jump matrix
$
J_R(z):=R_-(z)^{-1}R_+(z)
$
on the contour $\S_R$.
For convenience, we introduce the two functions
\begin{eqnarray}\label{v}
  v(z):=-z\log c
\end{eqnarray}
and
\begin{eqnarray}\label{W}
  W(z):=2in\pi w(nz-\b/2)e^{nv(z)}=\frac{2in\pi \G(nz+\b/2)c^{-\b/2}}{\G(nz+1-\b/2)}.
\end{eqnarray}
The jump conditions of $R(z)$ is given below.

\begin{prop}\label{prop-JoR}
On the contour $\S_R$, the jump matrix $J_R(z):=R_-(z)^{-1}R_+(z)$ has the following explicit expressions.
For $z=1+i\im z$ with $\im z\in(0,\pm\d)$, we have
\begin{eqnarray}\label{JoR1}
  J_R(z)=\left(\begin{matrix}
    1-e^{\pm2i\pi(nz-\b/2)}&\cfrac{W(z)e^{-nv(z)}\prod\limits_{j=0}^{n-1}(z-X_j)^2}
    {2i\sin(n\pi z-\b\pi/2)e^{\mp i\pi(nz-\b/2)}}\\
    \\
    -\cfrac{2i\sin(n\pi z-\b\pi/2)e^{\pm i\pi(nz-\b/2)}}{W(z)e^{-nv(z)}\prod\limits_{j=0}^{n-1}(z-X_j)^2}&1
  \end{matrix}\right).
\end{eqnarray}
On the positive real line, we have
\begin{subequations}\label{JoR2}
  \begin{equation}
    J_R(x)=\left(\begin{matrix}
    1&0\\
    \\
    \cfrac{4\sin^2(n\pi x-\b\pi/2)}{W(x)e^{-nv(x)}\prod\limits_{j=0}^{n-1}(x-X_j)^2}&1
  \end{matrix}\right)
  \end{equation}
  for $x\in(0,1)$, and
  \begin{equation}
    J_R(x)=\left(\begin{matrix}
    1&-W(x)e^{-nv(x)}\prod\limits_{j=0}^{n-1}(x-X_j)^2\\
    \\
    0&1
  \end{matrix}\right)
  \end{equation}
  for $x\in(1,\infty)$.
\end{subequations}
Furthermore, we have
\begin{subequations}\label{JoR3}
  \begin{equation}
    J_R(z)=\left(\begin{matrix}
    1&0\\
    \\
    \cfrac{2i\sin(n\pi z-\b\pi/2)e^{\pm i\pi(nz-\b/2)}}{W(z)e^{-nv(z)}\prod\limits_{j=0}^{n-1}(z-X_j)^2}&1
  \end{matrix}\right)
  \end{equation}
  for $z\in(0,\pm i\d)\cup(\pm i\d, 1\pm i\d)$, and
  \begin{equation}
    J_R(z)=\left(\begin{matrix}
    1&\cfrac{W(z)e^{-nv(z)}\prod\limits_{j=0}^{n-1}(z-X_j)^2}{2i\sin(n\pi z-\b\pi/2)e^{\mp i\pi(nz-\b/2)}}\\
    \\
    0&1
  \end{matrix}\right)
  \end{equation}
  for $z=\re z\pm i\d$ with $\re z\in(1,\infty)$.
\end{subequations}
\end{prop}
\begin{proof}
  For $z=1+i\im z$ with $\im z\in(0,\d)$, we obtain from (\ref{R}) and (\ref{W}) that
\begin{eqnarray*}
  R_+(z)=
    Q(z)\left(\begin{matrix}
    1&0\\
    \\
    -\cfrac{2ie^{i\pi(nz-\b/2)}\sin(n\pi z-\b\pi/2)}{W(z)e^{-nv(z)}\prod\limits_{j=0}^{n-1}(z-X_j)^2}&1
  \end{matrix}\right),
\end{eqnarray*}
and
\begin{eqnarray*}
  R_-(z)=
    Q(z)\left(\begin{matrix}
    1&-\cfrac{W(z)e^{-nv(z)}\prod\limits_{j=0}^{n-1}(z-X_j)^2}{2ie^{-i\pi (nz-\b/2)}\sin(n\pi z-\b\pi/2)}\\
    \\
    0&1
  \end{matrix}\right).
\end{eqnarray*}
Thus, we have
\begin{eqnarray*}
  J_R(z)&=&R_-(z)^{-1}R_+(z)
  \\&=&
    \left(\begin{matrix}
    1&\cfrac{W(z)e^{-nv(z)}\prod\limits_{j=0}^{n-1}(z-X_j)^2}{2ie^{-i\pi (nz-\b/2)}\sin(n\pi z-\b\pi/2)}\\
    \\
    0&1
  \end{matrix}\right)\left(\begin{matrix}
    1&0\\
    \\
    -\cfrac{2ie^{i\pi(nz-\b/2)}\sin(n\pi z-\b\pi/2)}{W(z)e^{-nv(z)}\prod\limits_{j=0}^{n-1}(z-X_j)^2}&1
  \end{matrix}\right)
  \\&=&
  \left(\begin{matrix}
    1-e^{2i\pi(nz-\b/2)}&\cfrac{W(z)e^{-nv(z)}\prod\limits_{j=0}^{n-1}(z-X_j)^2}{2ie^{-i\pi (nz-\b/2)}\sin(n\pi z-\b\pi/2)}\\
    \\
    -\cfrac{2ie^{i\pi(nz-\b/2)}\sin(n\pi z-\b\pi/2)}{W(z)e^{-nv(z)}\prod\limits_{j=0}^{n-1}(z-X_j)^2}&1
  \end{matrix}\right).
\end{eqnarray*}
Similarly, for $z\in 1-i(0,\d)$, we obtain from (\ref{R}) and (\ref{W}) that
\begin{eqnarray*}
  J_R(z)&=&
    \left(\begin{matrix}
    1&\cfrac{W(z)e^{-nv(z)}\prod\limits_{j=0}^{n-1}(z-X_j)^2}{2ie^{i\pi (nz-\b/2)}\sin(n\pi z-\b\pi/2)}\\
    \\
    0&1
  \end{matrix}\right)\left(\begin{matrix}
    1&0\\
    \\
    -\cfrac{2ie^{-i\pi(nz-\b/2)}\sin(n\pi z-\b\pi/2)}{W(z)e^{-nv(z)}\prod\limits_{j=0}^{n-1}(z-X_j)^2}&1
  \end{matrix}\right)
  \\&=&
  \left(\begin{matrix}
    1-e^{-2i\pi(nz-\b/2)}&\cfrac{W(z)e^{-nv(z)}\prod\limits_{j=0}^{n-1}(z-X_j)^2}{2ie^{i\pi (nz-\b/2)}\sin(n\pi z-\b\pi/2)}\\
    \\
    -\cfrac{2ie^{-i\pi(nz-\b/2)}\sin(n\pi z-\b\pi/2)}{W(z)e^{-nv(z)}\prod\limits_{j=0}^{n-1}(z-X_j)^2}&1
  \end{matrix}\right).
\end{eqnarray*}
Hence, the formula (\ref{JoR1}) is proved.

  For any $x>0$ with $x\notin \bb{X}$, we obtain from (\ref{R}) and (\ref{W}) that
\begin{eqnarray*}
  R_\pm(x)=
    Q(x)\left(\begin{matrix}
    1&0\\
    \\
    -\cfrac{2ie^{\pm i\pi(nx-\b/2)}\sin(n\pi x-\b\pi/2)}{W(x)e^{-nv(x)}\prod\limits_{j=0}^{n-1}(x-X_j)^2}&1
  \end{matrix}\right)
\end{eqnarray*}
for $x\in(0,1)$, and
\begin{eqnarray*}
  R_\pm(x)=Q(x)\left(\begin{matrix}
    1&-\cfrac{W(x)e^{-nv(x)}\prod\limits_{j=0}^{n-1}(x-X_j)^2}{2ie^{\mp i\pi (nx-\b/2)}\sin(n\pi x-\b\pi/2)}\\
    \\
    0&1
  \end{matrix}\right)
\end{eqnarray*}
for $x\in(1,\infty)$.
  A simple calculation gives (\ref{JoR2}).
  Since $R(z)$ has no singularity at $\bb{X}$ (Lemma \ref{lem-R}),
  formula (\ref{JoR2}) remains valid when $x\in \bb{X}$.

  Finally, coupling (\ref{R}) and (\ref{W}) yields (\ref{JoR3}) immediately.
  This completes our proof.
\end{proof}

\begin{prop}\label{prop-R}
The matrix-valued function $R(z)$ defined in (\ref{R}) is the unique solution of the
following Riemann-Hilbert problem:
\begin{enumerate}[(R1)]
\item $R(z)$ is analytic in
  $\mathbb{C}\setminus \S_R$;
\item for $z\in\S_R$,
  $
    R_+(z)=R_-(z)J_R(z),
  $ where the jump matrix $J_R(z)$ is given in Proposition \ref{prop-JoR};
\item for $z\in\mathbb{C}\setminus \S_R$, $R(z)=I+O(|z|^{-1})$ as $z\to\infty$.
\end{enumerate}
\end{prop}
\begin{proof}
  Condition (R1) follows from the analyticity condition (Q1) in Proposition \ref{prop-Q} and the definition of $R(z)$ in (\ref{R}).
  Proposition \ref{prop-JoR} gives (R2).
  Furthermore, the normalization condition (Q3) in Proposition \ref{prop-Q} yields (R3).
  The uniqueness of solution is again a direct consequence of Liouville's theorem.
\end{proof}

\section{The Equilibrium Measure}
For the preparation of the third transformation $R\to S$,
we investigate the equilibrium measure corresponding to the Meixner polynomials.
In the existing literature, the equilibrium measure is usually obtained by
solving a minimization problem of a certain quadratic functional (cf. \cite{BKMM03, BKMM07, DW07, DKMVZ99}).
Here, we prefer to use the method introduced by A. B. J. Kuijlaars and W. Van Assche \cite{KV99}.

Consider the monic polynomials $q_{n,N}(x):=N^{-n}\pi_n(Nx-\b/2)$, where $N\in\bb{N}$.
From (\ref{monic recurrence relation}), we have
\begin{eqnarray*}
  xq_{n,N}(x)=q_{n+1,N}(x)+\frac{(n+\b/2)(1+c)}{N(1-c)}q_{n,N}(x)+\frac{n(n+\b-1)c}{N^2(1-c)^2}q_{n,N}(x).
\end{eqnarray*}
The coefficients $\frac{(n+\b/2)(1+c)}{N(1-c)}$ and $\frac{n(n+\b-1)c}{N^2(1-c)^2}$ correspond to
the recurrence coefficients $b_{n,N}$ and $a_{n,N}^2$ in \cite[(1.6)]{KV99}.
Suppose $n/N\to t>0$ as $n\to\infty$. It can be shown that
\begin{eqnarray*}
  \frac{(n+\b/2)(1+c)}{N(1-c)}\to \frac{1+c}{1-c}t,\indent \sqrt{\frac{n(n+\b-1)c}{N^2(1-c)^2}}\to \frac{\sqrt c}{1-c}t.
\end{eqnarray*}
Define two constants
\begin{eqnarray}\label{ab}
  a:=\frac{1-\sqrt{c}}{1+\sqrt{c}}\indent\AND\indent
  b:=\frac{1+\sqrt{c}}{1-\sqrt{c}},
\end{eqnarray}
and note that $ab=1$.
The functions $\a(t)$ and $\b(t)$ in \cite[(1.8)]{KV99} are equal to $at$ and $bt$ respectively.
Therefore, from Theorem 1.4 in \cite{KV99},
the asymptotic zero distribution of $q_{n,N}(x)$ with $n/N\to t>0$ is given by
\begin{eqnarray*}
  \mu_t(x)=\frac 1 t\int_0^t\omega_{{[as, bs]}}(x)ds,
\end{eqnarray*}
where
\begin{eqnarray*}
  \cfrac{d\omega_{{[as, bs]}}(x)}{dx}=
  \begin{cases}
  \cfrac1{\pi\sqrt{(bs-x)(x-as)}},&\indent x\in(as,bs),\\
  \\
  0,&\indent \textup{elsewhere};
  \end{cases}
\end{eqnarray*}
see \cite[(1.4)]{KV99}. Thus, the density function of $\mu_t(x)$ is
$$
\frac{d\mu_t(x)}{dx}=\cfrac 1{\pi t}\int_{ax}^{bx}\cfrac{ds}{\sqrt{(bs-x)(x-as)}}
$$
for $x\in[0,at]$, and
$$
\frac{d\mu_t(x)}{dx}=\frac 1{\pi t}\int_{ax}^t\cfrac{ds}{\sqrt{(bs-x)(x-as)}}
$$
for $x\in[at,bt]$.
We only need to consider the special case $N=n$.
Therefore, when $t=1$, the density function becomes
\begin{eqnarray}\label{mu}
  \rho(x):=\frac{d\mu_1(x)}{dx}=\begin{cases}
    1,&\indent 0<x<a,\\
  \\
  \cfrac{1}{\pi }\arccos\cfrac{x(b+a)-2}{x(b-a)},&\indent a<x<b.
  \end{cases}
\end{eqnarray}
This {\it equilibrium measure} for our problem is $d\mu_1(x)=\rho(x)dx$.
Note that the constants $a$ and $b$ defined in (\ref{ab})
are the same as the constants $\a_-$ and $\a_+$ in \cite[(2.6)]{JW98}.
They are called the {\it Mhaskar-Rakhmanov-Saff numbers} or the {\it turning points}.
We now define the so-called {\it $g$ -- function}.
\begin{eqnarray}\label{g}
  g(z):=\int_0^b
  \log(z-x)\rho(x)dx
\end{eqnarray}
for $z\in\bb{C}\setminus(-\infty,b]$.
On account of (\ref{mu}), the derivative of $g(z)$ can be calculated as shown below.
\begin{eqnarray}\label{g'}
  g'(z)=\int_0^b
  \frac{1}{z-x}\rho(x)dx=
  -\log\cfrac{z(b+a)-2+2\sqrt{(z-a)(z-b)}}{z(b-a)}+\cfrac {-\log c}2.
\end{eqnarray}

\begin{prop}\label{prop-g'}
The function $g'(z)$ given in (\ref{g'}) is the unique solution of the following scalar Riemann-Hilbert problem:
\begin{enumerate}[(g1)]
\item  $g'(z)$ is analytic in
  $\mathbb{C}\setminus [0,b]$;
\item
denoting the limiting value taken by $g'(z)$ on the real line from the upper half plane by $g'_+(x)$
and that taken from the lower half plane by $g'_-(x)$,
the function $g'(z)$ satisfies the jump conditions:
  \begin{eqnarray}
    g'_+(x)-g'_-(x)=-2\pi i,&\indent& 0<x<a,\label{Jog'1}\\
    g'_+(x)+g'_-(x)=-\log c,&\indent& a<x<b;\label{Jog'2}
  \end{eqnarray}
\item
  $g'(z)=\cfrac 1 z+O(|z|^{-2})$, as $z\to\infty$.
\end{enumerate}
\end{prop}

\begin{proof}
  The analyticity condition (g1) is trivial by (\ref{g'}).
  The normalization condition (g3) follows from the fact that $\int_0^b\rho(x)dx=1$.
  For $0<x<a$, we obtain from (\ref{g'}) that
  $$g'_\pm(x)=-\log\cfrac{-x(b+a)+2+2\sqrt{(a-x)(b-x)}}{x(b-a)}\mp i\pi+\cfrac {-\log c}2.$$
  Therefore, the relation (\ref{Jog'1}) follows.
  For $a<x<b$, we obtain from (\ref{g'}) that
  $$g'_\pm(x)=-\log\cfrac{x(b+a)-2\pm 2i\sqrt{(x-a)(b-x)}}{x(b-a)}+\cfrac {-\log c}2.$$
  Therefore, the relation (\ref{Jog'2}) follows.
  Finally, the uniqueness is again guaranteed by Liouville's theorem.
\end{proof}

\begin{rem}\label{rem-mu}
  From (\ref{mu}) we observe that the equilibrium measure of the Meixner polynomials corresponds to
  the saturated-band-void configuration defined in \cite{BKMM07}; see also \cite{DW07}.
  We point out that the equilibrium measure $\rho(x)dx$ can be solved in a different way,
  that is, regard $\rho(x)dx$ as the measure which satisfies the constraint
  $$0\le\rho(x)\le1$$ on the interval $[0,\infty)$, and which minimizes the quadratic functional
  $$\int_0^\infty\int_0^\infty\log\frac{1}{|x-y|}\rho(x)\rho(y)dxdy+\int_0^\infty v(x)\rho(x)dx,$$
  where $v(x)$ is defined in (\ref{v}); see \cite{DS97,ST97}.
  Following the procedure in \cite[Section B.3]{BKMM07}, we first show that
  the Mhaskar-Rakhmanov-Saff numbers $a$ and $b$ are the solutions to the following equations
  \begin{eqnarray*}
    \int_a^b\frac{v'(x)}{\sqrt{(x-a)(b-x)}}dx-\int_0^a\frac{2\pi}{\sqrt{(a-x)(b-x)}}dx&=&0,\\
    \int_a^b\frac{xv'(x)}{\sqrt{(x-a)(b-x)}}dx-\int_0^a\frac{2\pi x}{\sqrt{(a-x)(b-x)}}dx&=&2\pi.
  \end{eqnarray*}
  In the second step we find that the function $g'(z)$,
  which corresponds to the function $F(z)$ in \cite[(710)]{BKMM07},
  has the explicit expression
  $$g'(z)=\int_0^a\frac{\sqrt{(z-a)(z-b)}}{\sqrt{(a-x)(b-x)}}\frac{dx}{x-z}
  -\int_a^b\frac{\sqrt{(z-a)(z-b)}}{\sqrt{(x-a)(b-x)}}\frac{v'(x)dx}{2\pi(x-z)}.$$
  Finally, the equilibrium measure $\rho(x)dx$ is supported on the interval $[0,b]$ and
  $$\rho(x)=\frac{g'_-(x)-g'_+(x)}{2\pi i}$$ for $x\in[0,b]$.
  A direct calculation shows that $\rho(x)=1$ on the saturated interval $[0,a]$,
  and $\rho(x)=\frac{1}{\pi }\arccos\frac{x(b+a)-2}{x(b-a)}$ on the band $[a,b]$.
  This agrees with the formula (\ref{mu}).
\end{rem}
Recall that $v(z)=-z\log c$ in (\ref{v}). It is easily seen from (\ref{g'}) that
$$-g'(\z)+\frac{v'(\z)}{2}=\log\frac{\z(b+a)-2+2\sqrt{(\z-a)(\z-b)}}{\z(b-a)}$$
for $\z\in\bb{C}\setminus(-\infty,b]$.
We introduce the so-called {\it $\phi$ -- function}.
\begin{eqnarray}\label{phi}
  \phi(z):=\int_b^z (-g'(\z)+\frac{v'(\z)}{2})d\z
  =\int_b^z\log\frac{\z(b+a)-2+2\sqrt{(\z-a)(\z-b)}}{\z(b-a)}d\z
\end{eqnarray}
for $z\in\bb{C}\setminus(-\infty,b]$. From the definition we observe
$$\phi(z)=-g(z)+v(z)/2+g(b)-v(b)/2=-g(z)+v(z)/2+l/2,$$
 where
\begin{eqnarray}\label{l}
l:=2g(b)-v(b)=2\log\frac{b-a}4-2
\end{eqnarray}
is called the {\it Lagrange multiplier}.
We also introduce the so-called {\it $\wt\phi$ -- function}.
\begin{eqnarray}\label{phi'}
  \wt\phi(z):=\int_a^z\log\frac{-\z(b+a)+2-2\sqrt{(\z-a)(\z-b)}}{\z(b-a)}d\z
  =\phi(z)\pm i\pi(1-z)
\end{eqnarray}
for $z\in\bb{C}_\pm$.
Note that the function $\wt\phi(z)$ can be analytically continued to the interval $(0,a)$;
see (\ref{phi'+-}).
We now provide some important properties of the $g$ --, $\phi$ -- and $\wt\phi$ -- functions.
\begin{prop}\label{prop-phi}
Let the functions $g, \phi, \wt\phi$ be defined as in (\ref{g}), (\ref{phi}) and (\ref{phi'}), respectively.
Recall from (\ref{v}) and (\ref{l}) that $v(z)=-z\log c$ and $l=2\log\frac{b-a}4-2$. We have
\begin{eqnarray}\label{gphi}
  2g(z)+2\phi(z)-v(z)-l=0
\end{eqnarray}
for all $z\in\bb{C}\setminus(-\infty,b]$.
Denote the boundary value taken by $\phi(z)$ on the real line from the upper half plane by $\phi_+$
and that taken from the lower half plane by $\phi_-$.
We have
\begin{eqnarray}\label{phi+-}
  \phi_+=\begin{cases}
    \phi_--2i\pi(1-x)&:\indent 0<x<a,\\
    -\phi_-&:\indent a<x<b,\\
    \phi_-&:\indent x>b.
  \end{cases}
\end{eqnarray}
Denote the boundary value taken by $\wt\phi(z)$ on the real line from the upper half plane by $\wt\phi_+$
and that taken from the lower half plane by $\wt\phi_-$.
We have
\begin{eqnarray}\label{phi'+-}
  \wt\phi_+=\begin{cases}
    \wt\phi_-&:\indent 0<x<a,\\
    -\wt\phi_-&:\indent a<x<b,\\
    \wt\phi_-+2i\pi(1-x)&:\indent x>b.
  \end{cases}
\end{eqnarray}
Denote the boundary value taken by $g(z)$ on the real line from the upper half plane by $g_+$
and that taken from the lower half plane by $g_-$.
We have
\begin{eqnarray}\label{prop-g+-}
  g_++g_--v-l=\begin{cases}
    -2\phi_+-2i\pi(1-x)&:\indent 0<x<a,\\
    0&:\indent a<x<b,\\
    -2\phi&:\indent x>b.
  \end{cases}
\end{eqnarray}
Furthermore, we have
\begin{eqnarray}\label{g+-}
  g_+-g_-=\begin{cases}
    2i\pi(1-x)&:\indent 0<x<a,\\
    -2\phi_+=2\phi_-&:\indent a<x<b,\\
    0&:\indent x>b.
  \end{cases}
\end{eqnarray}
For any small $\ve>0$ and $z\in U(b,\ve):=\{z\in\bb{C}: |z-b|<\ve\}$, we have
\begin{eqnarray}\label{prop-phi11}
  \phi(z)=\frac{4(z-b)^{3/2}}{3b\sqrt{b-a}}+O(\ve^2).
\end{eqnarray}
For any small $\ve>0$ and $z\in U(a,\ve):=\{z\in\bb{C}: |z-a|<\ve\}$, we have
\begin{eqnarray}\label{prop-phi12}
  \wt\phi(z)=\frac{-4(a-z)^{3/2}}{3a\sqrt{b-a}}+O(\ve^2).
\end{eqnarray}
For any small $\ve>0$ and $x>b+\ve$, we have
\begin{eqnarray}\label{prop-phi13}
  \phi(x)>\phi(b+\ve)=\frac{4\ve^{3/2}}{3b\sqrt{b-a}}+O(\ve^2).
\end{eqnarray}
For any small $\ve>0$ and $0<x<a-\ve$, we have
\begin{eqnarray}\label{prop-phi14}
  \wt\phi(x)<\wt\phi(a-\ve)=\frac{-4\ve^{3/2}}{3a\sqrt{b-a}}+O(\ve^2).
\end{eqnarray}
For any $x\in(a,b)$ and sufficiently small $y>0$, we have
\begin{eqnarray}
  \re\phi(x\pm iy)&=&-y\arccos\cfrac{x(b+a)-2}{x(b-a)}+O(y^2),\label{prop-phi21}\\
  \re\wt\phi(x\pm iy)&=&y\arccos\cfrac{2-x(b+a)}{x(b-a)}+O(y^2).\label{prop-phi22}
\end{eqnarray}
For any $x\in(b,\infty)$ and sufficiently small $y>0$, we have
\begin{eqnarray}\label{prop-phi23}
  \re\phi(x\pm iy)=\phi(x)+O(y^2),\indent
  \re\wt\phi(x\pm iy)=\phi(x)+\pi y+O(y^2).
\end{eqnarray}
For any $x\in(0,a)$ and sufficiently small $y>0$, we have
\begin{eqnarray}\label{prop-phi24}
  \re\wt\phi(x\pm iy)=\wt\phi(x)+O(y^2),\indent
  \re\phi(x\pm iy)=\wt\phi(x)-\pi y+O(y^2).
\end{eqnarray}
\end{prop}

\begin{proof}
 The relation (\ref{gphi}) follows from the definition of $\phi$ -- function in (\ref{phi}) and Lagrange multiplier in (\ref{l}).

 To prove (\ref{phi+-}), we first see from (\ref{phi}) that $\phi(z)$ is analytic for $z\in\bb{C}\setminus(-\infty,b]$.
 Thus, we have $\phi_+(x)-\phi_-(x)=0$ for $x>b$.
 Since
 $$\phi_\pm(x)=\int_b^x\log\frac{s(b+a)-2\pm2i\sqrt{(s-a)(b-s)}}{s(b-a)}ds$$
 for $a<x<b$, we also have $\phi_+(x)+\phi_-(x)=0$ for $a<x<b$.
 On the other hand, for $0<x<a$, we obtain from (\ref{phi}) that
 \begin{eqnarray*}
 \phi_\pm(x)&=&\int_b^a\log\frac{s(b+a)-2\pm2i\sqrt{(s-a)(b-s)}}{s(b-a)}ds\\
 &&+\int_a^x(\log\frac{-s(b+a)+2+2i\sqrt{(a-s)(b-s)}}{s(b-a)}\pm i\pi)ds.
 \end{eqnarray*}
 In view of the equality
 $$\int_b^a\log\frac{s(b+a)-2\pm2i\sqrt{(s-a)(b-s)}}{s(b-a)}ds
 =\mp i\int_a^b\arccos\cfrac{s(b+a)-2}{s(b-a)}ds=\mp i\pi(1-a),$$
 we have
 $$\phi_+(x)-\phi_-(x)=-2i\pi(1-a)+2i\pi(x-a)=-2i\pi(1-x)$$
 for $0<x<a$.
 This ends the proof of (\ref{phi+-}).

 Applying (\ref{phi'}) to (\ref{phi+-}) gives (\ref{phi'+-}).

 From (\ref{gphi}) we have $$g_+(x)+g_-(x)-v(x)-l=-\phi_+(x)-\phi_-(x)$$ for $x\in\bb{R}$.
 Hence, the relation (\ref{prop-g+-}) follows immediately from (\ref{phi+-}).

 It is easily seen from (\ref{g}) that the function $g(z)$ is analytic for $z\in\bb{C}\setminus(-\infty,b]$.
 Coupling (\ref{gphi}) and (\ref{phi+-}) yields
 $$g_+-g_-=\phi_--\phi_+=-2\phi_+=2\phi_-$$
 for $a<x<b$.
 On the other hand, a combination of (\ref{phi}), (\ref{gphi}) and (\ref{phi+-}) gives
 \begin{eqnarray*}
 g_+(a)-g_-(a)&=&\phi_-(a)-\phi_+(a)=2\phi_-(a)=2\int_b^a\log\frac{s(b+a)-2-2i\sqrt{(s-a)(b-s)}}{s(b-a)}ds\\
 &=&2i\int_a^b\arccos\cfrac{s(b+a)-2}{s(b-a)}ds=2i\pi(1-a).
 \end{eqnarray*}
 Coupling this with (\ref{Jog'1}) gives
 $$g_+(x)-g_-(x)=g_+(a)-g_-(a)+2i\pi(a-x)=2i\pi(1-x)$$
 for $0<x<a$.
 This completes the proof of (\ref{g+-}).

 For any small $\ve>0$ and $z\in U(b,\ve):=\{z\in\bb{C}: |z-b|<\ve\}$, from (\ref{phi}) we have
\begin{eqnarray*}
  \phi(z)=\int_b^z \log\left(1+\cfrac{2\sqrt{(b-a)(\z-b)}}{b(b-a)}+O(\ve)\right)d\z
  =\frac{4(z-b)^{3/2}}{3b\sqrt{b-a}}+O(\ve^2).
\end{eqnarray*}
 Here again, we have used the fact that $ab=1$.
 This gives (\ref{prop-phi11}).

 For any small $\ve>0$ and $z\in U(a,\ve):=\{z\in\bb{C}: |z-a|<\ve\}$, from (\ref{phi'}) we have
\begin{eqnarray*}
  \phi(z)=\int_a^z \log\left(1+\cfrac{2\sqrt{(a-\z)(b-a)}}{a(b-a)}+O(\ve)\right)d\z
  =\frac{-4(a-z)^{3/2}}{3a\sqrt{b-a}}+O(\ve^2).
\end{eqnarray*}
 This gives (\ref{prop-phi12}).

 From (\ref{phi}) and (\ref{phi'}), we have
 $$\phi'(x)=\log\cfrac{x(b+a)-2+2\sqrt{(x-a)(x-b)}}{x(b-a)}>0$$ for $x>b$ and
 $$\wt\phi'(x)=\log\cfrac{-\z(b+a)+2-2\sqrt{(\z-a)(\z-b)}}{\z(b-a)}>0$$ for $0<x<a$.
 Consequently, $\phi(x)>\phi(b+\ve)$ for $x>b+\ve$, and $\wt\phi(x)<\wt\phi(a-\ve)$ for $0<x<a-\ve$.
 Therefore, the formulas (\ref{prop-phi13}) and (\ref{prop-phi14}) follow from (\ref{prop-phi11}) and (\ref{prop-phi12}), respectively.

 It is easily seen from (\ref{phi}) and (\ref{phi'}) that
 $\phi_\pm(x)$ and $\wt\phi_\pm(x)$ are purely imaginary for $a<x<b$.
 Hence, for any $x\in(a,b)$ and sufficiently small $y>0$, we have
 \begin{eqnarray*}
  \re\phi(x\pm iy)&=&\re\int_x^{x\pm iy} \log\frac{\z(b+a)-2+2\sqrt{(\z-a)(\z-b)}}{\z(b-a)}d\z\\
  &=&\re\int_x^{x\pm iy} \left(\pm i\arccos\cfrac{x(b+a)-2}{x(b-a)}+O(y)\right)d\z\\
  &=&-y\arccos\cfrac{x(b+a)-2}{x(b-a)}+O(y^2),
 \end{eqnarray*}
 and
 \begin{eqnarray*}
  \re\wt\phi(x\pm iy)&=&\re\int_x^{x\pm iy} \log\frac{-\z(b+a)+2-2\sqrt{(\z-a)(\z-b)}}{\z(b-a)}d\z\\
  &=&\re\int_x^{x\pm iy} \left(\mp i\arccos\cfrac{2-x(b+a)}{x(b-a)}+O(y)\right)d\z\\
  &=&y\arccos\cfrac{2-x(b+a)}{x(b-a)}+O(y^2).
 \end{eqnarray*}
 This ends the proof of (\ref{prop-phi21}) and (\ref{prop-phi22}).

 For any $x\in(b,\infty)$ and sufficiently small $y>0$, from (\ref{phi}) we have
\begin{eqnarray*}
  \re\phi(x\pm iy)&=&\phi(x)+\re\int_x^{x\pm iy} \log\frac{\z(b+a)-2+2\sqrt{(\z-a)(\z-b)}}{\z(b-a)}d\z\\
  &=&\phi(x)+\re\int_x^{x\pm iy} \left(\log\frac{x(b+a)-2+2\sqrt{(x-a)(x-b)}}{x(b-a)}+O(y)\right)d\z\\
  &=&\phi(x)+O(y^2).
\end{eqnarray*}
 Moreover, we obtain from (\ref{phi'}) that
\begin{eqnarray*}
  \re\wt\phi(x\pm iy)=\re\phi(x\pm iy)+\pi y=\phi(x)+\pi y+O(y^2).
\end{eqnarray*}
 This proves (\ref{prop-phi23}).

 For any $x\in(0,a)$ and sufficiently small $y>0$, from (\ref{phi'}) we have
\begin{eqnarray*}
  \re\wt\phi(x\pm iy)&=&\wt\phi(x)+\re\int_x^{x\pm iy} \log\frac{-\z(b+a)+2-2\sqrt{(\z-a)(\z-b)}}{\z(b-a)}d\z\\
  &=&\wt\phi(x)+\re\int_x^{x\pm iy} \left(\log\frac{-x(b+a)+2+2\sqrt{(a-x)(b-x)}}{x(b-a)}+O(y)\right)d\z\\
  &=&\wt\phi(x)+O(y^2).
\end{eqnarray*}
 Moreover, we obtain from (\ref{phi'}) that
\begin{eqnarray*}
  \re\phi(x\pm iy)=\re\wt\phi(x\pm iy)-\pi y=\wt\phi(x)-\pi y+O(y^2).
\end{eqnarray*}
 This proves (\ref{prop-phi24}).
\end{proof}

\begin{rem}\label{rem-d}
  Recall that the constant $\d_0>0$ introduced in the definition of $R(z)$ has not been determined;
  see (\ref{R}). Fix any $0<c<1$ and $1\le\b<2$,
  we choose $\d_0>0$ to be sufficiently small such that
  the function $\phi(z)^{2/3}$ is analytic in the open ball $U(b,\d_0)$ and
  the function $\wt\phi(z)^{2/3}$ is analytic in the open ball $U(a,\d_0)$.
  We also require $\d_0$ to be so small that
  the formulas (\ref{prop-phi11})-(\ref{prop-phi24}) in Proposition \ref{prop-phi}
  are valid whenever $\ve, y\in(0,\d_0)$.
  The existence of such a positive constant $\d_0$ is obvious.
  Furthermore, since the functions $\phi(z)$ and $\wt\phi(z)$ depend only on the constants $c$ and $\b$.
  the constant $\d_0$ is independent of the polynomial degree $n$.
\end{rem}

For the sake of simplicity, we introduce some auxiliary functions. Define
\begin{eqnarray}\label{E}
  E(z):=\left(\frac {z-1}z\right)^{\frac{1-\b}2}\exp\left\{-n\int_0^{1}\log(z-x)dx\right\}
  \prod\limits_{k=0}^{n-1}(z-X_k)
\end{eqnarray}
for $z\in\mathbb{C}\setminus[0, 1]$, and
\begin{eqnarray}\label{E'}
  \wt E(z):=\cfrac{\pm i  E(z)e^{\mp i\pi (nz-\b/2)}}{2\sin(n\pi z-\b\pi/2)}
\end{eqnarray}
for $z\in\mathbb{C}_\pm$, and
\begin{eqnarray}\label{H}
  H(z):=\left(\frac z{z-1}\right)^{1-\b}W(z)
\end{eqnarray}
for $z\in\mathbb{C}\setminus[0, 1]$, and
\begin{eqnarray}\label{H'}
  \wt H(z)&:=&\left(\frac z{1-z}\right)^{1-\b}W(z)
\end{eqnarray}
for $z\in\mathbb{C}\setminus(-\infty, 0]\cup[1, \infty)$,
where $W(z)$ is defined in (\ref{W}). We also recall from (\ref{X}) that $X_k=\frac{k+\b/2}n$.
The properties of the above auxiliary functions are given in the following lemma.

\begin{lem}\label{lem-E}
The function $\wt E(z)$ defined in (\ref{E'}) can be analytically continued to the interval $(0,1)$.
Moreover, for any $0<x<1$, we have
\begin{eqnarray}\label{E+E-}
 \wt E(x)^2=\cfrac{E_+(x)E_-(x)}{4\sin^2(n\pi x-\b\pi/2)}.
\end{eqnarray}
For any $z\in\bb{C_\pm}$, we have
\begin{eqnarray}
  E(z)/\wt E(z)&=&\mp 2i e^{\pm i\pi (nz-\b/2)}\sin(n\pi z-\b\pi/2)=1-e^{\pm2i\pi(nz-\b/2)},\label{E/E'}\\
  \wt H(z)&=&H(z)e^{\pm i\pi(1-\b)}=-H(z)e^{\mp i\pi\b}.\label{H/H'}
\end{eqnarray}
As $n\to\infty$, we have $E(z)\sim 1$ uniformly for $z$ bounded away from the interval $[0, 1]$
and $E(z)/\wt E(z)\sim1$ uniformly for $z$ bounded away from the real line.
\end{lem}

\begin{proof}
  For $0<x<1$, from (\ref{E}) we have
  $$E_\pm(x)=\left(\frac {1-x}x\right)^{\frac{1-\b}2}e^{\pm i\pi(1-\b)/2}\exp\left\{-n\int_0^1\log|x-s|ds\right\}
  e^{\mp n\pi i(1-x)}\prod\limits_{k=0}^{n-1}(z-X_k).$$
  Consequently, we obtain $E_+(x)/E_-(x)=-e^{2i\pi(nx-\b/2)}$.
  Therefore, it is readily seen from (\ref{E'}) that $\wt E_+(x)=\wt E_-(x)$ on the interval $(0,1)$.
  Moreover, we have
  $$\wt E^2(x)=\cfrac{E_+(x)E_-(x)}{4\sin^2(n\pi x-\b\pi/2)},\indent 0<x<1.$$
  This gives (\ref{E+E-}).

  The relation (\ref{E/E'}) follows from (\ref{E'}).
  The relation (\ref{H/H'}) follows from (\ref{H}) and (\ref{H'}).

  Let $z$ be bounded away from the interval $[0,1]$. Using Stirling's formula, we have
  \begin{eqnarray*}
  \prod\limits_{k=0}^{n-1}(z-X_k)&=&\prod\limits_{k=0}^{n-1}\left(z-\cfrac{k+\b/2}n\right)
  =\cfrac{\G(nz-\b/2+1)}{n^n\G(nz-\b/2-n+1)}\\
  &\sim&\cfrac{\sqrt{2\pi(nz-\b/2)}(\frac{nz-\b/2}e)^{nz-\b/2}}{n^n\sqrt{2\pi(nz-\b/2-n)}(\frac{nz-\b/2-n}e)^{nz-\b/2-n}}\\
  &=&\cfrac{(nz-\b/2)^{\frac{1-\b}2}(\frac{nz-\b/2}{nz})^{nz}(nz)^{nz}}
  {n^n(nz-\b/2-n)^{\frac{1-\b}2}(\frac{nz-\b/2-n}{nz-n})^{nz-n}(nz-n)^{nz-n}e^n}\\
  &\sim&\left(\frac z{z-1}\right)^{\frac{1-\b}2}\left(\frac z{z-1}\right)^{nz}\left(\cfrac{z-1}{e}\right)^n,
  \end{eqnarray*}
  as $n\to\infty$.
  In view of the equality
  $$\exp\left\{-n\int_0^{1}\log(z-x)dx\right\}=\frac{e^n(z-1)^{nz}}{z^{nz}(z-1)^n},$$
  we then obtain from (\ref{E}) that
  \begin{eqnarray*}
    E(z)&\sim&\left(\frac {z-1}z\right)^{\frac{1-\b}2}\cfrac{e^n(z-1)^{nz}}{z^{nz}(z-1)^n}
    \left(\frac z{z-1}\right)^{\frac{1-\b}2}\left(\frac z{z-1}\right)^{nz}\left(\cfrac{z-1}{e}\right)^n=1,
  \end{eqnarray*}
  as $n\to\infty$.

  Finally, as $n\to\infty$, it is easily seen from (\ref{E/E'}) that
  $E(z)/\wt E(z)\sim1$ uniformly for $z$ bounded away from the real line.
  This ends the proof of the lemma.
\end{proof}

\section{The Third Transformation $R\to S$}

Recalling the definition of $g(z)$ in (\ref{g}), we introduce the function
\begin{eqnarray}\label{G}
  G(z):=ng(z)-n\int_0^1\log(z-x)dx=n\int_0^b\log(z-x)\rho(x)dx-n\int_0^1\log(z-x)dx.
\end{eqnarray}
Since $\int_0^b\rho(x)dx=1$, it is easily seen that $G(z)=O(|z|^{-1})$ as $z\to\infty$.
Furthermore, the function $G(z)$ is analytic in $\bb{C}\setminus(-\infty,b]$.
Applying (\ref{phi'}) and (\ref{g+-}) to (\ref{G}) implies
\begin{eqnarray}\label{G+-}
  G_+-G_-=\begin{cases}
    0&:\indent -\infty<x<a,\\
    -2n\wt\phi_+=2n\wt\phi_-&:\indent a<x<1,\\
    -2n\phi_+=2n\phi_-&:\indent 1<x<b.
  \end{cases}
\end{eqnarray}
Note that $G(z)$ can be analytically continued to the interval $(-\infty,a)$.
In terms of the function $G(z)$, we now make the third transformation
\begin{eqnarray}\label{S}
  S(z):=e^{(-nl/2)\s_3}R(z)e^{(-G(z)+nl/2)\s_3}.
\end{eqnarray}
To compute the jump conditions of $S(z)$, we first state the following lemma.

\begin{lem}\label{lem-EH}
For $0<x<1$, we have
\begin{eqnarray}\label{lem-EH1}
\cfrac{4\sin^2(n\pi x-\b\pi/2)}{W\exp(G_++G_--nv-nl)}\prod\limits_{k=0}^{n-1}(x-X_k)^{-2}
=\frac{e^{n(\wt \phi_++\wt\phi_-)}}{\wt H\wt E^2}.
\end{eqnarray}
For $x>1$, we have
\begin{eqnarray}\label{lem-EH2}
-We^{G_++G_--nv-nl}\prod\limits_{k=0}^{n-1}(x-X_k)^2
=\frac{-HE^2}{e^{n(\phi_++\phi_-)}}.
\end{eqnarray}
For $z\in \bb{C}_\pm$, we have
\begin{eqnarray}
  \cfrac{e^{\pm i\pi(nz-\b/2)}We^{2G-nv-nl}}{2i\sin(n\pi z-\b\pi/2)}\prod\limits_{k=0}^{n-1}(z-X_k)^2
  =\cfrac{\pm\wt H \wt EE}{e^{2n\wt\phi}},
\label{lem-EH3}\\
  \cfrac{-2i\sin(n\pi z-\b\pi/2)}{e^{\mp i\pi(nz-\b/2)}We^{2G-nv-nl}}\prod\limits_{k=0}^{n-1}(z-X_k)^{-2}
  =\cfrac{e^{2n\phi}}{\pm H\wt EE}.\label{lem-EH4}
\end{eqnarray}
\end{lem}

\begin{proof}
  Combining (\ref{E}) and (\ref{G}) gives
  $$Ee^{-G}=e^{-ng}\left(\frac {z-1}z\right)^{\frac{1-\b}2}\ \prod\limits_{k=0}^{n-1}(z-X_k),\indent z\in\bb{C_\pm}.$$
  Therefore, we have
  $$E_+E_-e^{-(G_++G_-)}=e^{-n(g_++g_-)}\left(\frac {1-x}x\right)^{1-\b}\ \prod\limits_{k=0}^{n-1}(x-X_k)^2,\indent x\in(0,1).$$
  Imposing (\ref{E+E-}) and (\ref{H'}), we then obtain
  $$\cfrac{4\sin^2(n\pi x-\b\pi/2)}{W\exp(G_++G_--nv-nl)}\prod\limits_{k=0}^{n-1}(x-X_k)^{-2}
  =\frac{e^{-n(g_++g_--v-l)}}{\wt H\wt E^2},\indent x\in(0,1).$$
  Therefore, the equality (\ref{lem-EH1}) follows from (\ref{phi'}) and (\ref{gphi}).

  For $x>1$, combining (\ref{G}) and (\ref{E}) yields
  $$e^{G_++G_--nv-nl}/E^2=e^{n(g_++g_--v-l)}\left(\frac x{x-1}\right)^{1-\b}\ \prod\limits_{k=0}^{n-1}(x-X_k)^{-2}.$$
  Imposing (\ref{gphi}) and (\ref{H}), we then obtain (\ref{lem-EH2}) immediately.

  For $z\in\bb{C_\pm}$, combining (\ref{G}), (\ref{gphi}) and (\ref{E}) gives
  \begin{eqnarray}\label{lem-EH01}
  \cfrac{e^{\pm i\pi(nz-\b/2)}We^{2G-nv-nl}}{2i\sin(n\pi z-\b\pi/2)}\prod\limits_{k=0}^{n-1}(z-X_k)^2
  =\cfrac{e^{-2n\phi\pm i\pi(nz-\b/2)}WE^2}{2i\sin(n\pi z-\b\pi/2)}\left(\frac z{z-1}\right)^{1-\b}.
  \end{eqnarray}
  From (\ref{phi'}), we have
  \begin{eqnarray}\label{lem-EH02}
    e^{-2n\phi}=e^{-2n\wt\phi\pm2in\pi(1-z)}.
  \end{eqnarray}
  From (\ref{H}) and (\ref{H/H'}), we have
  \begin{eqnarray}\label{lem-EH03}
    W\left(\frac z{z-1}\right)^{1-\b}=H=-\wt He^{\pm i\pi\b}.
  \end{eqnarray}
  From (\ref{E'}), we have
  \begin{eqnarray}\label{lem-EH04}
    \cfrac{E}{2i\sin(n\pi z-\b\pi/2)}=\mp\wt E e^{\pm i\pi(nz-\b/2)}.
  \end{eqnarray}
  Therefore, applying (\ref{lem-EH02})-(\ref{lem-EH04}) to (\ref{lem-EH01}) gives (\ref{lem-EH3}).

  For $z\in\bb{C_\pm}$, combining (\ref{G}), (\ref{gphi}) and (\ref{E}) gives
  $$
  \cfrac{-2i\sin(n\pi z-\b\pi/2)}{e^{\mp i\pi(nz-\b/2)}We^{2G-nv-nl}}\prod\limits_{k=0}^{n-1}(z-X_k)^{-2}
  =\cfrac{-2i\sin(n\pi z-\b\pi/2)}{e^{-2n\phi\mp i\pi(nz-\b/2)}WE^2}\left(\frac {z-1}z\right)^{1-\b}.
  $$
  Hence, it is easy to obtain (\ref{lem-EH4}) using (\ref{E'}) and (\ref{H}).
\end{proof}

Now, we come back to the transformation (\ref{S}).
It is easily seen from (R1) and (\ref{G+-}) that the matrix-valued function $S(z)$ is analytic in $\bb{C}\setminus\S_R$.
Let $\S_S:=\S_R$ be the oriented contour depicted in Figure \ref{fig-Q2R}.
We calculate the jump matrices for $S(z)$ in the following proposition.

\begin{prop}\label{prop-JoS}
On the contour $\S_S$, the jump matrix $J_S(z):=S_-(z)^{-1}S_+(z)$ has the following explicit expressions.
For $0<x<a$, we have
\begin{eqnarray}\label{JoS11}
J_S(x)=\left(\begin{matrix}
    1&0\\
    \\
    \cfrac{e^{2n\wt\phi}}{\wt H\wt E^2}&1
  \end{matrix}\right).
\end{eqnarray}
For $a<x<1$, we have
\begin{eqnarray}\label{JoS12}
J_S(x)=\left(\begin{matrix}
    e^{-2n\wt\phi_-}&0\\
    \\
    \cfrac{1}{\wt H\wt E^2}&e^{-2n\wt\phi_+}
  \end{matrix}\right).
\end{eqnarray}
For $1<x<b$, we have
\begin{eqnarray}\label{JoS13}
J_S(x)=\left(\begin{matrix}
    e^{2n\phi_+}&- H E^{2}\\
    \\
    0&e^{2n\phi_-}
  \end{matrix}\right).
\end{eqnarray}
For $x>b$, we have
\begin{eqnarray}\label{JoS14}
J_S(x)=\left(\begin{matrix}
    1&\cfrac{- H E^{2}}{e^{2n\phi}}\\
    \\
    0&1
  \end{matrix}\right).
\end{eqnarray}
For $z=1+i\im z$ with $\im z\in(0,\pm\d)$, we have
\begin{eqnarray}\label{JoS21}
J_S(z)=\left(\begin{matrix}
    {E}/{\wt E}&\cfrac{\pm\wt H \wt EE}{e^{2n\wt\phi}}\\
    \\
    \cfrac{e^{2n\phi}}{\pm H \wt EE}&1
  \end{matrix}\right).
\end{eqnarray}
For $z\in (0,\pm i\d)\cup(\pm i\d, 1\pm i\d)$, we have
\begin{eqnarray}\label{JoS22}
J_S(z)=\left(\begin{matrix}
    1&0\\
    \\
    \cfrac{e^{2n\phi}}{\mp H \wt EE}&1
  \end{matrix}\right).
\end{eqnarray}
For $z=\re z\pm i\d$ with $\re z\in (1,\infty)$, we have
\begin{eqnarray}\label{JoS23}
J_S(z)=\left(\begin{matrix}
    1&\cfrac{\pm\wt H \wt EE}{e^{2n\wt\phi}}\\
    \\
    0&1
  \end{matrix}\right).
\end{eqnarray}
The jump conditions of $S(z)$ on the contour $\S_S$ are illustrated in Figure \ref{fig-JoS}.
\end{prop}

\begin{figure}[htp]
\centering
\includegraphics{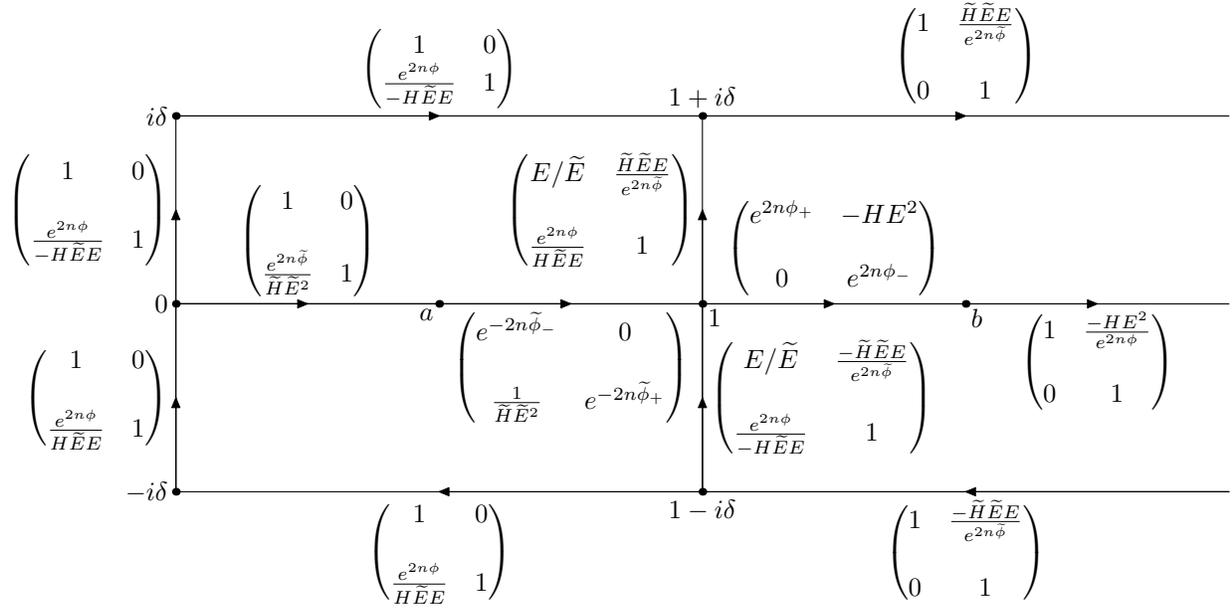}
\caption{The jump conditions of $S(z)$ on the contour $\S_S$.}\label{fig-JoS}
\end{figure}

\begin{proof}
  From (\ref{S}), we have
  \begin{eqnarray}\label{JoS-pf1}
  J_S(z)=e^{(G_-(z)-nl/2)\s_3}J_R(z)e^{(-G_+(z)+nl/2)\s_3}.
  \end{eqnarray}
  Combining (\ref{JoR2}), (\ref{lem-EH1}), (\ref{lem-EH2}) and (\ref{JoS-pf1}) implies
  \begin{subequations}\label{JoS-pf2}
    \begin{equation}
      J_S(x)=\left(\begin{matrix}
    e^{G_--G_+}&0\\
    \\
    \cfrac{e^{n(\wt \phi_++\wt\phi_-)}}{\wt H\wt E^2}&e^{G_+-G_-}
  \end{matrix}\right)
    \end{equation}
    for $x\in(0,1)$, and
    \begin{equation}
      J_S(x)=\left(\begin{matrix}
    e^{G_--G_+}&\cfrac{-HE^2}{e^{n(\phi_++\phi_-)}}\\
    \\
    0&e^{G_+-G_-}
  \end{matrix}\right)
    \end{equation}
    for $x\in(1,\infty)$.
  \end{subequations}
  Applying (\ref{phi+-}), (\ref{phi'+-}) and (\ref{G+-}) to (\ref{JoS-pf2}) gives
  (\ref{JoS11})-(\ref{JoS14}) immediately.

  Recall that the function $G(z)$ is analytic in $\bb{C}\setminus[a,b]$.
  A combination of (\ref{JoR1}), (\ref{JoR3}), (\ref{E/E'}), (\ref{lem-EH3}), (\ref{lem-EH4}) and (\ref{JoS-pf1}) gives
  (\ref{JoS21})-(\ref{JoS23}) immediately.
\end{proof}

\begin{prop}\label{prop-S}
The matrix-valued function $S(z)$ defined in (\ref{S}) is the unique solution of the
following Riemann-Hilbert problem:
\begin{enumerate}[(S1)]
\item  $S(z)$ is analytic in
  $\mathbb{C}\setminus \S_S$:
\item for $z\in\S_S$,
  $
    S_+(z)=S_-(z)J_S(z),
  $ where $J_S(z)$ is given in Proposition \ref{prop-JoS};
\item for $z\in\mathbb{C}\setminus \S_S$, $S(z)=I+O(|z|^{-1})$ as $z\to\infty$.
\end{enumerate}
\end{prop}
\begin{proof}
  The analyticity condition (S1) is clear from the definition of $S(z)$ in (\ref{S}),
  and from the analyticity condition (R1) of $R(z)$ in Proposition \ref{prop-R}.
  The jump condition (S2) is proved in Proposition \ref{prop-JoS}.
  Furthermore, the normalization condition (R3) of $R(z)$ in Proposition \ref{prop-R} gives (S3).
  The uniqueness is again a direct consequence of Liouville's theorem.
\end{proof}

\section{The Final Transformation $S\to T$}
For $a<x<1$, we can factorize the jump matrix $J_S(x)$ in  (\ref{JoS12}) as below
\begin{eqnarray*}
  J_S(x)
  =\left(\begin{matrix}
    \wt E\ &\cfrac{\wt H\wt E}{e^{2n\wt\phi_-}}\\
    \\
    0&1/{\wt E }
  \end{matrix}\right)\left(\begin{matrix}
    0&-\wt H\\
    1/\wt H&0
  \end{matrix}\right)
  \left(\begin{matrix}
    1/{\wt E }\ &\cfrac{\wt H\wt E}{e^{2n\wt\phi_+}}\\
    \\
    0&\wt E
  \end{matrix}\right),
\end{eqnarray*}
where we have used (\ref{phi'+-}). Similarly, by using (\ref{phi+-}),
for $1<x<b$ we can factorize the jump matrix $J_S(x)$ in (\ref{JoS13}) as below
\begin{eqnarray*}
  J_S(x)=\left(\begin{matrix}
    E &0\\
    \\
    \cfrac{e^{2n\phi_-}}{-H E}\ &1/E
  \end{matrix}\right)\left(\begin{matrix}
    0&- H\\
    1/ H&0
  \end{matrix}\right)\left(\begin{matrix}
    1/E &0\\
    \\
    \cfrac{e^{2n\phi_+}}{-H E}\ &E
  \end{matrix}\right).
\end{eqnarray*}
This suggests the final transformation $S\to T$ defined by
\begin{subequations}\label{T}
  \begin{equation}\label{T1}
    T(z):=S(z)\wt E ^{\s_3}
  \end{equation}
  for $z\in\Omega_{T,\pm}^1$, and
  \begin{equation}
    T(z):=S(z)\left(\begin{matrix}
     \wt E\ &\cfrac{\mp\wt H \wt E}{e^{2n\wt\phi}}\\
     \\
    0&1/\wt E
  \end{matrix}\right)
  \end{equation}
  for $z\in\Omega_{T,\pm}^2$, and
  \begin{equation}
    T(z):=S(z)\left(\begin{matrix}
     E &0\\
    \\
    \cfrac{e^{2n\phi}}{\pm H E }\ &1/E
  \end{matrix}\right)
  \end{equation}
  for $z\in\Omega_{T,\pm}^3$, and
  \begin{equation}
    T(z):=S(z) E ^{\s_3}
  \end{equation}
  for $z\in\Omega_{T,\pm}^4\cup\Omega_T^\infty$,
\end{subequations}
where the domain $\Omega_T=\Omega_{T,\pm}^1\cup\cdots\cup\Omega_{T,\pm}^4\cup\Omega_T^\infty$ is depicted in Figure \ref{fig-SoT}.
For easy reference, we have used Figure \ref{fig-S2T} to illustrate the definition of the transformation $S\to T$.
\begin{figure}[htp]
\centering
\includegraphics{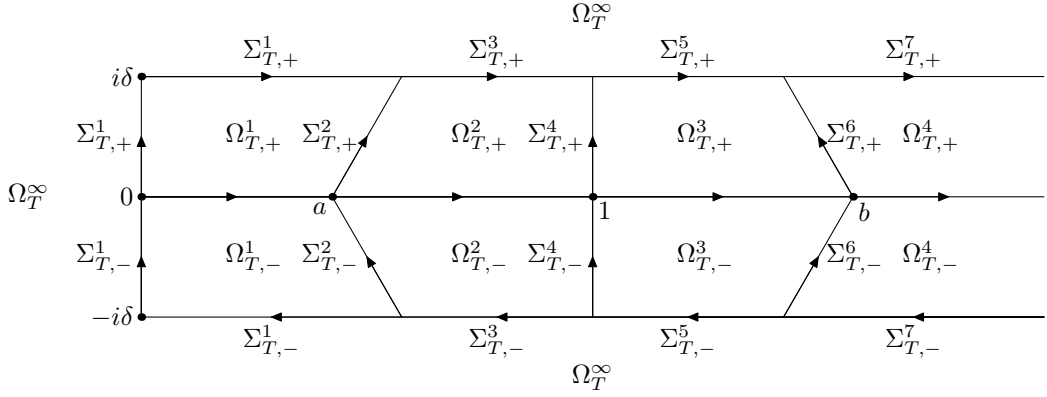}
\caption{The region $\Omega_T$ and
 the contour $\S_T$.
}\label{fig-SoT}
\end{figure}

\begin{figure}[htp]
\centering
\includegraphics{S2T.eps}
\caption{The transformation $S\to T$.}\label{fig-S2T}
\end{figure}

\noindent The following proposition gives the jump conditions of $T(z)$ on the contour $\S_T$,
where $\S_T=\S_{T,\pm}^1\cup\cdots\cup\S_{T,\pm}^7\cup(0,\infty)$; see Figure \ref{fig-SoT}.
\begin{prop}\label{prop-JoT}
On the contour $\S_T$, the jump matrix $J_T(z):=T_-(z)^{-1}T_+(z)$ can be calculated as below.
For $z\in\S_{T,\pm}^4$, we have
\begin{eqnarray}\label{JoT4}
  J_T(z)=\left(\begin{matrix}
     {E}&0\\
    \\
    \cfrac{e^{2n\phi}}{\pm H E }\ &{1}/{E }
  \end{matrix}\right)^{-1}\left(\begin{matrix}
    {E}/{\wt E}&\cfrac{\pm\wt H \wt EE}{e^{2n\wt\phi}}\\
    \\
    \cfrac{e^{2n\phi}}{\pm H \wt EE}&1
  \end{matrix}\right)\left(\begin{matrix}
     \wt E\ &\cfrac{\mp\wt H \wt E}{e^{2n\wt\phi}}\\
     \\
    0&1/{\wt E}
  \end{matrix}\right)=I.
\end{eqnarray}
For $z\in\S_{T,\pm}^3$, we have
\begin{eqnarray}\label{JoT3}
  J_T(z)=\left(\begin{matrix}
     \wt E&\cfrac{\mp\wt H \wt E}{e^{2n\wt\phi}}\\
     \\
    0&1/{\wt E}
  \end{matrix}\right)^{-1}
    \left(\begin{matrix}
    1&0\\
    \\
    \cfrac{e^{2n\phi}}{\mp H \wt EE}&1
  \end{matrix}\right)E(z)^{\s_3}=\left(\begin{matrix}
    1&\cfrac{\pm\wt H \wt E}{e^{2n\wt\phi}E}\\
    \\
    \cfrac{e^{2n\phi}}{\mp H}&{\wt E}/{E}
  \end{matrix}\right).
\end{eqnarray}
For $z\in\S_{T,\pm}^5$, we have
\begin{eqnarray}\label{JoT5}
  J_T(z)=\left(\begin{matrix}
     E&0\\
    \\
    \cfrac{e^{2n\phi}}{\pm H E }&1/E
  \end{matrix}\right)^{-1}\left(\begin{matrix}
    1&\cfrac{\pm\wt H \wt EE}{e^{2n\wt\phi}}\\
    \\
    0&1
  \end{matrix}\right)E(z)^{\s_3}=\left(\begin{matrix}
    1&\cfrac{\pm\wt H \wt E}{e^{2n\wt\phi}E}\\
    \\
    \cfrac{e^{2n\phi}}{\mp H}&{\wt E}/{E}
  \end{matrix}\right).
\end{eqnarray}
For $z\in\S_{T,\pm}^1$, we have
\begin{eqnarray}\label{JoT1}
  J_T(z)=\wt E(z)^{-\s_3}
  \left(\begin{matrix}
    1&0\\
    \\
    \cfrac{e^{2n\phi}}{\mp H \wt EE}&1
  \end{matrix}\right) E(z)^{\s_3}=\left(\begin{matrix}
    E/\wt E&0\\
    \\
    \cfrac{e^{2n\phi}}{\mp H}&\wt E/E
  \end{matrix}\right).
\end{eqnarray}
For $z\in\S_{T,\pm}^7$, we have
\begin{eqnarray}\label{JoT7}
  J_T(z)=E(z)^{-\s_3}
  \left(\begin{matrix}
    1&\cfrac{\pm\wt H \wt EE}{e^{2n\wt\phi}}\\
    \\
    0&1
  \end{matrix}\right)E(z)^{\s_3}=\left(\begin{matrix}
    1&\cfrac{\pm\wt H\wt E}{e^{2n\wt\phi}E}\\
    \\
    0&1
  \end{matrix}\right).
\end{eqnarray}
On the positive real line, we have
\begin{subequations}\label{JoT0}
  \begin{equation}
    J_T(x)=\left(\begin{matrix}
    1&0\\
    \\
    \cfrac{e^{2n\wt\phi}}{\wt H}&1
  \end{matrix}\right)
  \end{equation}
  for $0<x<a$, and
  \begin{equation}
    J_T(x)=\left(\begin{matrix}
    0&-\wt H \\
    1/\wt H&0
  \end{matrix}\right)
  \end{equation}
  for $a<x<1$, and
  \begin{equation}
    J_T(x)=\left(\begin{matrix}
    0&- H \\
    1/ H&0
  \end{matrix}\right)
  \end{equation}
  for $1<x<b$, and
  \begin{equation}
    J_T(x)=\left(\begin{matrix}
    1&\cfrac{- H}{e^{2n\phi}} \\
    \\
    0&1
  \end{matrix}\right)
  \end{equation}
  for $x>b$.
\end{subequations}
Furthermore, we have
\begin{subequations}\label{JoT2/6}
  \begin{equation}
    J_T(z)=\left(\begin{matrix}
    1&\cfrac{\pm\wt H}{e^{2n\wt\phi}}\\
    \\
    0&1
  \end{matrix}\right)
  \end{equation}
  for $z\in \S_{T,\pm}^2$, and
  \begin{equation}
    J_T(z)=\left(\begin{matrix}
     1&0\\
    \\
    \cfrac{e^{2n\phi}}{\pm H}&1
  \end{matrix}\right)
  \end{equation}
  for $z\in \S_{T,\pm}^6$.
\end{subequations}
The jump conditions of $T(z)$ on the contour $\S_T$ are illustrated in Figure \ref{fig-JoT}.
\end{prop}

\begin{figure}[hbtp]
\centering
\includegraphics{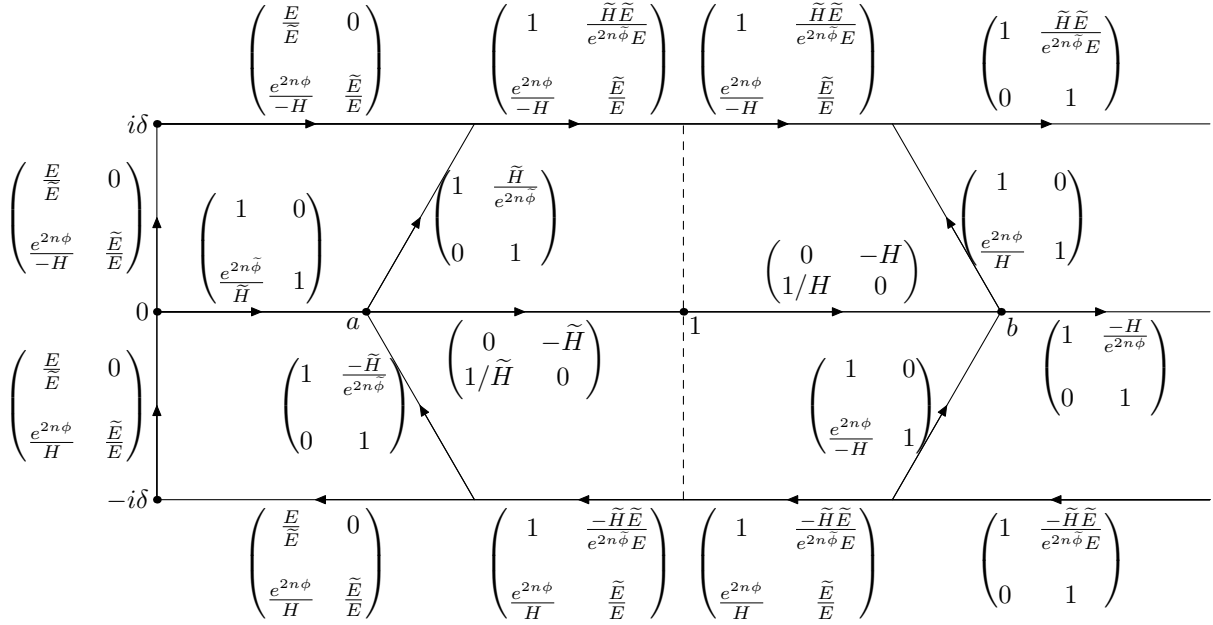}
\caption{
The jump conditions of $T(z)$.
The dashed line means that there is actually no
\newline \leftline{\hspace{1.7cm}
jump on this line.}
}\label{fig-JoT}
\end{figure}

\begin{proof}
  For $z\in\S_{T,\pm}^4$, we obtain from (\ref{JoS21}) and (\ref{T})
  \begin{eqnarray}\label{JoT41}
  J_T(z)=\left(\begin{matrix}
     {E}&0\\
    \\
    \cfrac{e^{2n\phi}}{\pm H E }&{1}/{E }
  \end{matrix}\right)^{-1}\left(\begin{matrix}
    {E}/{\wt E}&\cfrac{\pm\wt H \wt EE}{e^{2n\wt\phi}}\\
    \\
    \cfrac{e^{2n\phi}}{\pm H \wt EE}&1
  \end{matrix}\right)\left(\begin{matrix}
     \wt E&\cfrac{\mp\wt H \wt E}{e^{2n\wt\phi}}\\
     \\
    0&1/{\wt E}
  \end{matrix}\right).
  \end{eqnarray}
  A combination of (\ref{H/H'}), (\ref{E/E'}) and (\ref{phi'}) gives
  \begin{eqnarray}\label{JoT00}
  \cfrac{-e^{2n(\phi-\wt\phi)}\wt H}{HE}+\cfrac{1}{\wt E}=\cfrac{1}{E}.
  \end{eqnarray}
  Therefore, we have
  \begin{eqnarray}\label{JoT42}
  \left(\begin{matrix}
    {E}/{\wt E}&\cfrac{\pm\wt H \wt EE}{e^{2n\wt\phi}}\\
    \\
    \cfrac{e^{2n\phi}}{\pm H \wt EE}&1
  \end{matrix}\right)\left(\begin{matrix}
     \wt E&\cfrac{\mp\wt H \wt E}{e^{2n\wt\phi}}\\
     \\
    0&1/{\wt E}
  \end{matrix}\right)=\left(\begin{matrix}
     {E}&0\\
    \\
    \cfrac{e^{2n\phi}}{\pm H E }&{1}/{E }
  \end{matrix}\right).
  \end{eqnarray}
  Coupling (\ref{JoT41}) and (\ref{JoT42}) yields (\ref{JoT4}).

  For $z\in\S_{T,\pm}^3$, by applying (\ref{JoS22}) to (\ref{T}) we obtain
  \begin{eqnarray*}
  J_T(z)=\left(\begin{matrix}
     \wt E&\cfrac{\mp\wt H \wt E}{e^{2n\wt\phi}}\\
     \\
    0&1/{\wt E}
  \end{matrix}\right)^{-1}
    \left(\begin{matrix}
    1&0\\
    \\
    \cfrac{e^{2n\phi}}{\mp H \wt EE}&1
  \end{matrix}\right)E(z)^{\s_3}=\left(\begin{matrix}
    1&\cfrac{\pm\wt H \wt E}{e^{2n\wt\phi}E}\\
    \\
    \cfrac{e^{2n\phi}}{\mp H}&{\wt E}/{E}
  \end{matrix}\right).
  \end{eqnarray*}
  Here, we have made use of (\ref{JoT00}) in the second equality.
  Thus, formula (\ref{JoT3}) is proved.

  Similarly, by applying (\ref{JoS23}) and (\ref{JoT00}) to (\ref{T})
  we obtain (\ref{JoT5}).

  Moreover, applying (\ref{JoS22}) and (\ref{JoS23}) to (\ref{T}) gives, respectively, (\ref{JoT1}) and (\ref{JoT7}).

  For $0<x<a$, we obtain from (\ref{JoS11}) and (\ref{T})
  $$
  J_T(x)=\wt E(x)^{-\s_3}\left(\begin{matrix}
    1&0\\
    \\
    \cfrac{e^{2n\wt\phi}}{\wt H\wt E^2}&1
  \end{matrix}\right)\wt E(x)^{\s_3}=\left(\begin{matrix}
    1&0\\
    \\
    \cfrac{e^{2n\wt\phi}}{\wt H}&1
  \end{matrix}\right).
  $$
  For $x>b$, we obtain from (\ref{JoS14}) and (\ref{T})
  $$
  J_T(x)=E(x)^{-\s_3}\left(\begin{matrix}
    1&\cfrac{- HE^2}{e^{2n\phi}} \\
    \\
    0&1
  \end{matrix}\right)E(x)^{\s_3}=\left(\begin{matrix}
    1&\cfrac{- H}{e^{2n\phi}} \\
    \\
    0&1
  \end{matrix}\right).
  $$
  Similarly, for $a<x<1$, by applying (\ref{JoS12}) to (\ref{T}) we obtain
  $$
  J_T(x)=\left(\begin{matrix}
    \wt E &\cfrac{\wt H\wt E}{e^{2n\wt\phi_-}}\\
    \\
    0&1/{\wt E }
  \end{matrix}\right)^{-1}\left(\begin{matrix}
    e^{-2n\wt\phi_-}&0\\
    \\
    \cfrac{1}{\wt H\wt E^2}&e^{-2n\wt\phi_-}
  \end{matrix}\right)
  \left(\begin{matrix}
    \wt E &\cfrac{-\wt H\wt E}{e^{2n\wt\phi_+}}\\
    \\
    0&1/{\wt E }
  \end{matrix}\right)=\left(\begin{matrix}
    0&-\wt H\\
    1/\wt H&0
  \end{matrix}\right).
  $$
  For $1<x<b$, by applying (\ref{JoS13}) to (\ref{T}) we obtain
  $$
  J_T(x)=\left(\begin{matrix}
    E &0\\
    \\
    \cfrac{e^{2n\phi_-}}{- H E}&1/E
  \end{matrix}\right)^{-1}\left(\begin{matrix}
    e^{2n\phi_+}&- H E^{2}\\
    \\
    0&e^{2n\phi_-}
  \end{matrix}\right)\left(\begin{matrix}
    E &0\\
    \\
    \cfrac{e^{2n\phi_+}}{H E}&1/E
  \end{matrix}\right)=\left(\begin{matrix}
    0&- H\\
    1/ H&0
  \end{matrix}\right),
  $$
  thus proving (\ref{JoT0}).
  The second equalities in the last two equations actually follow from the first two
  equations at the beginning of Section 7.

  Finally, since $S(z)$ has no jump on $\S_{T,\pm}^2$ and $\S_{T,\pm}^6$,
  we obtain (\ref{JoT2/6}) from the definition of $T(z)$ in (\ref{T}).
  This ends the proof of the proposition.
\end{proof}

\begin{prop}\label{prop-T}
The matrix-valued function $T(z)$ defined in (\ref{T}) is the unique solution of the
following Riemann-Hilbert problem:
\begin{enumerate}[(T1)]
\item $T(z)$ is analytic in
  $\mathbb{C}\setminus \S_T$;
\item for $z\in\S_T$,
  $
    T_+(z)=T_-(z)J_T(z),
  $ where $J_T(z)$ is given in Proposition \ref{prop-JoT};
\item for $z\in\mathbb{C}\setminus \S_T$, $T(z)=I+O(|z|^{-1})$ as $z\to\infty$.
\end{enumerate}
\end{prop}
\begin{proof}
  The analyticity follows from (S1) in Proposition \ref{prop-S} and the definition of $T(z)$.
  Proposition \ref{prop-JoT} gives (T2).
  Furthermore, (S3) in Proposition \ref{prop-S} yields (T3).
  The uniqueness is again an immediate consequence of Liouville's theorem.
\end{proof}

\section{Construction of Parametrix}

With the aid of Figure \ref{fig-JoT}, we observe from (\ref{E/E'}) and Propositions
\ref{prop-phi} \& \ref{prop-JoT} that as $n\to\infty$, the jump matrix
$J_T(z)$ converges exponentially fast to the identity for
$z$ bounded away from $[a,b]\cup\{0\}$.
The limiting Riemann-Hilbert problem can be divided into several local problems,
whose solutions can be constructed explicitly.
Since these solutions to the local Riemann-Hilbert problems are not unique,
we shall choose as in \cite{DKMVZ99} some specific ones,
which are asymptotically equal to each other in the overlapping regions.
By piecing them together, we build a function that is defined in the whole complex plane.
This matrix-valued function is our desired parametrix.

We first consider the Riemann-Hilbert problem:
\begin{enumerate}[(M1)]
\item $M(z)$ is analytic in
  $\mathbb{C}\setminus [a,b]$;
\item  $M(z)$ satisfies the jump conditions
  \begin{eqnarray}\label{JoM}
  \begin{cases}
    M_+(x)=M_-(x)\left(\begin{matrix}
    0&-\wt H \\
    1/\wt H&0
  \end{matrix}\right)&:\indent a<x<1,\\
    M_+(x)=M_-(x)\left(\begin{matrix}
    0&- H \\
    1/ H&0
  \end{matrix}\right)&:\indent 1<x<b;
  \end{cases}
  \end{eqnarray}
\item
  $M(z)=I+O(|z|^{-1})$, as $z\to\infty$.
\end{enumerate}
Recall that
$H(z)=[z/(z-1)]^{1-\b}W(z)$ and $\wt H(z)=[z/(1-z)]^{1-\b}W(z)$, where
$$
W(z)=\frac{2ni\pi \G(nz+\b/2)c^{-\b/2}}{\G(nz+1-\b/2)};
$$
see (\ref{W}), (\ref{H}) and (\ref{H'}).
Define
\begin{eqnarray}
  V(z):=\log\frac{\G(nz+1-\b/2)}{z^{1-\b}\G(nz+\b/2)}-\log(2ni\pi c^{-\b/2}).
  \label{V}
\end{eqnarray}
Clearly,
\begin{eqnarray}
  H(z)=(z-1)^{\b-1}e^{-V(z)},\indent \wt H(z)=(1-z)^{\b-1}e^{-V(z)}.
  \label{HV}
\end{eqnarray}
From the Stirling series \cite[(6.1.40) and (6.3.18)]{AS70}, we have
$$\log \G(z)=(z-\frac 1 2)\log z-z+\frac 1 2\log(2\pi)+O(|z|^{-1}),\indent
\frac{\G'(z)}{\G(z)}=\log z-\frac 1{2z}+O(|z|^{-2})
$$
as $z\to\infty$.
The estimate holds uniformly for $z$ bounded away from the negative
real line. Thus, we obtain the double asymptotic behavior for
$V(z)$ as $n\to\infty$ or $z\to\infty$,
\begin{eqnarray}
  \label{V-asymp}
  V(z)=-\b\log n-\log(2i\pi c^{-\b/2})+O(\frac1{n|z|}),\indent
  V'(z)=O(\frac1{n|z|^2}),
\end{eqnarray}
which again holds uniformly for $z$ bounded away from the negative real line.
For $z$ bounded away from $(-\infty,0]\cup\{1\}$,
it follows from (\ref{HV}) and (\ref{V-asymp}) that
\begin{eqnarray}
  \label{H-asymp1}
  |n^{-\b}H(z)|+|n^\b H(z)^{-1}|+|n^{-\b}\wt H(z)|+|n^\b\wt H(z)^{-1}|=O(1)
\end{eqnarray}
as $n\to\infty$.
Furthermore, for $\re z\ge0$, we have from (\ref{W}) and Stirling's formula that
$W(z)^{-1}$ is uniformly bounded as $n\to\infty$.
Thus, from (\ref{H}) and (\ref{H'}), we obtain
\begin{eqnarray}
  \label{H-asymp2}
  |H(z)^{-1}|+|\wt H(z)^{-1}|=O(1)
\end{eqnarray}
uniformly for $\re z\ge0$ and $z\ne1$.
Here, we have used the assumption $1\le\b<2$.
We remark that formula (\ref{H-asymp2}) will later be used
in the proof of Proposition \ref{prop-JoK}.
Now, we introduce the function
\begin{eqnarray}
  \label{G'}
  \wt G(z):=-\int_z^\infty\int_a^b\frac{V'(s)\sqrt{(s-a)(b-s)}}{2\pi(s-\z)\sqrt{(\z-a)(\z-b)}}dsd\z.
\end{eqnarray}
\begin{lem}
\label{lem-G'}
  The function $\wt G(z)$ defined in (\ref{G'}) is a solution to the Riemann-Hilbert problem:
  \begin{enumerate}[(G1)]
    \item $\wt G(z)$ is analytic in $\mathbb{C}\setminus [a,b]$;
    \item for $x\in(a,b)$, $\wt G(z)$ satisfies the jump condition
    \begin{eqnarray}
    \label{JoG'}
    \wt G_+(x)+\wt G_-(x)-V(x)-L=0,
    \end{eqnarray}
    where $L:=2\wt G(b)-V(b)$ is a constant independent of $x$;
    \item $\wt G(z)=O(|z|^{-1})$, as $z\to\infty$.
  \end{enumerate}
  As $n\to\infty$, we have
  \begin{eqnarray}
    \label{G'-asymp}
    \wt G(z)=O(1/n)
  \end{eqnarray}
  uniformly for $z\in\mathbb{C}$.
  Here, the value of $\wt G(x)$ at $x\in(a,b)$ takes the meaning of boundary value from the upper or lower half-plane.
  Therefore, (\ref{G'-asymp}) implies that $|\wt G_+(x)|+|\wt G_-(x)|=O(1/n)$ for $x\in(a,b)$.
  Furthermore, we have the asymptotic behavior for the constant $L$:
  \begin{eqnarray}
    \label{L-asymp}
    L=\b\log n+\log(2i\pi c^{-\b/2})+O(1/n).
  \end{eqnarray}
\end{lem}
\begin{proof}
From (\ref{G'}), we obtain
\begin{eqnarray}
  \label{G''}
  \wt G'(z)=\int_a^b\frac{V'(s)\sqrt{(s-a)(b-s)}ds}{2\pi(s-z)\sqrt{(z-a)(z-b)}}.
\end{eqnarray}
It is easily seen that $\wt G'(z)$ is analytic in $\mathbb{C}\setminus [a,b]$ and
$\wt G'_+(x)+\wt G'_-(x)=V'(x)$ for $x\in(a,b).$
Moreover, $\wt G'(z)=O(|z|^{-2})$ as $z\to\infty$.
Thus, (G1)-(G3) follows.

From (\ref{V-asymp}) and (\ref{G''}), we have $(1+|z|^2)|\wt G'(z)|=O(1/n)$ as $n\to\infty$.
This estimate is uniform for $z\in\mathbb{C}$. Therefore, $\wt G(z)=O(1/n)$, thus giving (\ref{G'-asymp}).

Since $L=2\wt G(b)-V(b)$,
formula (\ref{L-asymp}) follows from (\ref{V-asymp}) and (\ref{G'-asymp}).
\end{proof}
With the aid of the function $\wt G(z)$, we now solve the Riemann-Hilbert problem (M1)-(M3) explicitly.
\begin{prop}
  The Riemann-Hilbert problem (M1)-(M3) has a solution given by
  \begin{eqnarray}
  \label{M}
    M(z)=\left(\begin{matrix}
    \cfrac{(z-1)^{\frac{1-\b}2}(\frac{\sqrt{z-a}+\sqrt{z-b}}2)^\b}{(z-a)^{1/4}(z-b)^{1/4} e^{-\wt G(z)}}&
    \cfrac{-i(z-1)^{\frac{\b-1}2} (\frac{\sqrt{z-a}-\sqrt{z-b}}2)^\b}{(z-a)^{1/4}(z-b)^{1/4} e^{\wt G(z)-L}} \\
    \\
    \cfrac{i(z-1)^{\frac{1-\b}2}(\frac{\sqrt{z-a}-\sqrt{z-b}}2)^{2-\b}}{(z-a)^{1/4}(z-b)^{1/4}e^{L-\wt G(z)}}&
    \cfrac{(z-1)^{\frac{\b-1}2} (\frac{\sqrt{z-a}+\sqrt{z-b}}2)^{2-\b}}{(z-a)^{1/4}(z-b)^{1/4}e^{\wt G(z)}}
  \end{matrix}\right).
  \end{eqnarray}
\end{prop}
\begin{proof}
  Since $\wt G(z)$ is analytic in $\mathbb{C}\setminus [a,b]$, the entries of $M(z)$
  can be analytically continued to the interval $(-\infty, a)$. Thus, (M1) follows.

  The jump conditions in (M2) can be verified as below.
  For $x\in(1,b)$, we obtain from (\ref{HV}) and (\ref{M}) that
  \begin{eqnarray*}
  M_\pm^{11}(x)&=&\cfrac{(x-1)^{\frac{1-\b}2}(\frac{\sqrt{x-a}\pm i\sqrt{b-x}}2)^\b}
  {(x-a)^{1/4}(b-x)^{1/4}e^{\pm i\pi/4} e^{-\wt G_\pm(x)}},\\
  M_\mp^{12}(x)&=&\cfrac{-iH(x)(x-1)^{\frac{1-\b}2} (\frac{\sqrt{x-a}\pm i\sqrt{b-x}}2)^\b}
  {(x-a)^{1/4}(b-x)^{1/4}e^{\mp i\pi/4} e^{\wt G_\mp(x)-V(x)-L}}.
  \end{eqnarray*}
  Thus, the relation (\ref{JoG'}) implies that $M_\mp^{12}(x)/M_\pm^{11}(x)=\pm H(x)$ for $x\in(1,b)$.
  On the other hand, for $x\in(a,1)$, we have from (\ref{HV}) and (\ref{M})
  \begin{eqnarray*}
  M_\pm^{11}(x)&=&\cfrac{(1-x)^{\frac{1-\b}2}e^{\frac{\pm i\pi(1-\b)}2}(\frac{\sqrt{x-a}\pm i\sqrt{b-x}}2)^\b}
  {(x-a)^{1/4}(b-x)^{1/4}e^{\pm i\pi/4} e^{-\wt G_\pm (x)}},\\
  M_\mp^{12}(x)&=&\cfrac{-i\wt H(x)(1-x)^{\frac{1-\b}2}e^{\frac{\pm i\pi(1-\b)}2} (\frac{\sqrt{x-a}\pm i\sqrt{b-x}}2)^\b}
  {(x-a)^{1/4}(b-x)^{1/4}e^{\mp i\pi/4} e^{\wt G_\mp(x)-V(x)-L}}.
  \end{eqnarray*}
  Coupling this with (\ref{JoG'}) yields $M_\mp^{12}(x)/M_\pm^{11}(x)=\pm\wt H(x)$ for $x\in(a,1)$.
  Similarly, a combination of (\ref{HV}), (\ref{JoG'}) and (\ref{M}) gives
  $$\frac{M_\mp^{22}(x)}{M_\pm^{21}(x)}=\begin{cases}
    \pm H(x),&\indent x\in(1,b),\\
    \pm\wt H(x),&\indent x\in(a,1).
  \end{cases}
  $$
  This proves (M2).

  By (G3) in Lemma \ref{lem-G'}, we have $\wt G(z)=O(|z|^{-1})$ as $z\to\infty$.
  Hence, it is easily seen from (\ref{M}) that $M(z)=I+O(|z|^{-1})$ as $z\to\infty$.
\end{proof}

Note that the solution to the Riemann-Hilbert problem (M1)-(M3) is not unique
because the boundary conditions at the two end points $a$ and $b$ are not specified.
However, as we shall see,
the matrix-valued function $M(z)$ defined in (\ref{M}) seems to be the best choice for us.

From (\ref{V-asymp}), (\ref{G'-asymp}) and (\ref{L-asymp}) we have, as $n\to\infty$,
$|\wt G(z)|+|V(z)+L|=O(1/n)$
uniformly for $z$ bounded away from the negative real line.
By virtue of the relations
\begin{eqnarray*}
\sqrt{z-a}+\sqrt{z-b}=e^{\pm i\pi/2}(\sqrt{b-z}+\sqrt{a-z}),\\
\sqrt{z-a}-\sqrt{z-b}=e^{\mp i\pi/2}(\sqrt{b-z}-\sqrt{a-z}),
\end{eqnarray*}
we obtain from (\ref{HV}) and (\ref{M}) that
\begin{eqnarray*}
&&\wt H^{-\s_3/2}M\wt H^{\s_3/2}\\&=&\left(\begin{matrix}
    \cfrac{(1-z)^{\frac{1-\b}2}(\frac{\sqrt{b-z}+\sqrt{a-z}}2)^\b}{(b-z)^{1/4}(a-z)^{1/4}}&
    \cfrac{i(1-z)^{\frac{1-\b}2} (\frac{\sqrt{b-z}-\sqrt{a-z}}2)^\b}{(b-z)^{1/4}(a-z)^{1/4}} \\
    \\
    \cfrac{-i(1-z)^{\frac{\b-1}2}(\frac{\sqrt{b-z}-\sqrt{a-z}}2)^{2-\b}}{(b-z)^{1/4}(a-z)^{1/4}}&
    \cfrac{(1-z)^{\frac{\b-1}2} (\frac{\sqrt{b-z}+\sqrt{a-z}}2)^{2-\b}}{(b-z)^{1/4}(a-z)^{1/4}}
  \end{matrix}\right)\\
  &&\times\left(I+O(\frac 1 n)\right),
\end{eqnarray*}
which again holds uniformly for $z$ bounded away from the negative real line. Define
\begin{eqnarray}
  \label{m'}
  \wt m(z):=\cfrac{(1-z)^{\frac{1-\b}2\s_3}}{(b-z)^{1/4}(a-z)^{1/4}}
  \left(\begin{matrix}
    (\frac{\sqrt{b-z}+\sqrt{a-z}}2)^\b &
    i(\frac{\sqrt{b-z}-\sqrt{a-z}}2)^\b
    \\
    \\
    -i(\frac{\sqrt{b-z}-\sqrt{a-z}}2)^{2-\b}\ &
    (\frac{\sqrt{b-z}+\sqrt{a-z}}2)^{2-\b}
  \end{matrix}\right)
  \left(\begin{matrix}
    1&i \\
    i&1
  \end{matrix}\right).
\end{eqnarray}
As $n\to\infty$, we have
\begin{eqnarray}
  \label{Mm'}
  \wt H(z)^{-\s_3/2}M(z)\wt H(z)^{\s_3/2}
  \left(\begin{matrix}
    1&i \\
    i&1
  \end{matrix}\right)=\wt m(z)\left(I+O(\frac 1 n)\right).
\end{eqnarray}
Similarly, define
\begin{eqnarray}
  \label{m}
  m(z):=\cfrac{(z-1)^{\frac{1-\b}2\s_3}}{(z-a)^{1/4}(z-b)^{1/4}}
  \left(\begin{matrix}
    (\frac{\sqrt{z-a}+\sqrt{z-b}}2)^\b &
    -i(\frac{\sqrt{z-a}-\sqrt{z-b}}2)^\b \\
    \\
    i(\frac{\sqrt{z-a}-\sqrt{z-b}}2)^{2-\b}\ &
    (\frac{\sqrt{z-a}+\sqrt{z-b}}2)^{2-\b}
  \end{matrix}\right)
  \left(\begin{matrix}
    1&i \\
    i&1
  \end{matrix}\right).
\end{eqnarray}
From (\ref{HV}) and (\ref{M}), we obtain
\begin{eqnarray}
  \label{Mm}
  H(z)^{-\s_3/2}M(z)H(z)^{\s_3/2}
  \left(\begin{matrix}
    1&i \\
    i&1
  \end{matrix}\right)=m(z)\left(I+O(\frac 1 n)\right),\indent n\to\infty.
\end{eqnarray}
The estimates (\ref{Mm'}) and (\ref{Mm}) hold uniformly for $z$ bounded away from the negative real line.
Recall that we are using capital letters to emphasize the dependence on $n$;
see the paragraph before Proposition \ref{prop-P}.
The small letters $\wt m$ and $m$ in (\ref{m'}) and (\ref{m}), respectively,
indicate that these two matrix-valued functions are independent of $n$.
Note that for any small $\ve>0$, the matrix-valued function $m(z)(z-b)^{\s_3/4}$
is analytic in $U(b,\ve):=\{z\in\bb{C}: |z-b|<\ve\}$, and the matrix-valued function
$\wt m(z)(a-z)^{-\s_3/4}$ is analytic in $U(a,\ve):=\{z\in\bb{C}: |z-a|<\ve\}$.

Next, we find the solution to the scalar Riemann-Hilbert problem:
\begin{enumerate}[(D1)]
\item $D(z)$ is analytic in
  $\mathbb{C}\setminus (-i\infty, i\infty)$;
\item $D(z)$ satisfies the jump condition
  \begin{eqnarray}\label{JoD}
    D_+(z)=D_-(z)\frac {E(z)}{\wt E(z)},&\indent z\in (-i\infty, i\infty),
  \end{eqnarray}
  where the functions $D_+(z)$ and $D_-(z)$ denote the boundary values of $D(z)$
  taken from the left and right of the imaginary line respectively;
\item for $z\in\mathbb{C}\setminus (-i\infty, i\infty)$,
  $D(z)=1+O(|z|^{-1})$ as $z\to\infty$.
\end{enumerate}
Recall from (\ref{E/E'}) that
$E(z)/\wt E(z)=1-e^{\pm2i\pi(nz-\b/2)}.$
The solution to the Riemann-Hilbert problem (D1)-(D3) is given by
\begin{eqnarray}\label{D}
  D(z)&=&\exp\bigg\{\frac 1{2\pi i}\int_{-i\infty}^{i\infty}\log\bigg(\frac{E(\z)}{\wt E(\z)}\bigg)\frac{d\z}{\z-z}\bigg\}
  \nonumber\\
  &=&\exp\bigg\{\frac 1{2\pi i}\int_0^\infty
  \left[\frac{\log(1-e^{-2n\pi s-i\pi\b})}{s+iz}-\frac{\log(1-e^{-2n\pi s+i\pi\b})}{s-iz}\right]ds\bigg\}.
\end{eqnarray}
It can be shown that as $n\to\infty$, the function $D(z)$ converges uniformly
to one for $z$ bounded away from the origin.

Finally, we construct the parametrix $T_{par}(z)$.
We shall make use of the so-called {\it Airy parametrix}:
\begin{subequations}\label{A}
  \begin{equation}\label{A1}
    A(z):=\left(\begin{array}{cc}\Ai(z)&\omega^2\Ai(\omega^2z)\\
    \\i\Ai'(z)&i\omega \Ai'(\omega^2z)\end{array}\right)
  \end{equation}
  for $\arg z\in(0,2\pi/3)$, and
  \begin{equation}\label{A2}
    A(z):=\left(\begin{array}{cc}-\omega \Ai(\omega z)&\omega^2\Ai(\omega^2z)\\
    \\-i\omega^2\Ai'(\omega z)&i\omega \Ai'(\omega^2z)\end{array}\right)
  \end{equation}
  for $\arg z\in(2\pi/3,\pi)$, and
  \begin{equation}\label{A3}
    A(z):=\left(\begin{array}{cc}-\omega^2\Ai(\omega^2z)&-\omega \Ai(\omega z)\\
    \\-i\omega \Ai'(\omega^2z)&-i\omega^2 \Ai'(\omega z)\end{array}\right)
  \end{equation}
  for $\arg z\in(-\pi,-2\pi/3)$, and
  \begin{equation}\label{A4}
    A(z):=\left(\begin{array}{cc}\Ai(z)&-\omega \Ai(\omega z)\\
    \\i\Ai'(z)&-i\omega^2 \Ai'(\omega z)\end{array}\right)
  \end{equation}
  for $\arg z\in(-2\pi/3,0)$.
\end{subequations}
Let $\d_0$ be determined in Remark \ref{rem-d}.
Fix any $0<\ve<\d<\d_0$ and
denote by $U(z_0,\ve)$ the open disk centered at $z_0$ with radius $\ve$,
where $z_0=0,\ a \OR b$.
We define
\begin{eqnarray}\label{Tp}
 T_{par}(z):=M(z)
\end{eqnarray}
for $z\in\bb{C}\setminus(U(0,\ve)\cup U(a,\ve)\cup U(b,\ve))$, and
\begin{eqnarray}\label{Tp0}
 T_{par}(z):=M(z)D(z)^{\s_3}
\end{eqnarray}
for $z\in U(0,\ve)$, and
\begin{eqnarray}\label{Tpb}
 T_{par}(z):=
 \sqrt\pi  H(z)^{\s_3/2}m(z) F(z)^{\s_3/4} {A}(F(z))e^{n\phi(z)\sigma_3} H(z)^{-\s_3/2}
\end{eqnarray}
for $z\in U(b,\ve)$, and
\begin{eqnarray}\label{Tpa}
 T_{par}(z):=
 \sqrt\pi {\wt H(z)}^{\s_3/2}\wt m(z)\wt F(z)^{-\s_3/4}\s_1 {A}(\wt F(z))\s_1e^{n\wt\phi(z)\sigma_3} {\wt H(z)}^{-\s_3/2}
\end{eqnarray}
for $z\in U(a,\ve)$,
where the functions $F(z)$ and $\wt F(z)$ are defined by
\begin{eqnarray}\label{F}
 F(z):=\left(\frac 3 2n\phi(z)\right)^{2/3},\indent \wt F(z):=\left(-\frac 3 2n\wt\phi(z)\right)^{2/3},
\end{eqnarray}
and $\s_1:=\left(\begin{array}{cc}0&1\\ 1&0\end{array}\right)$ and
$\s_3:=\left(\begin{array}{cc}1&0\\ 0&-1\end{array}\right)$ are the Pauli matrices.
By virtue of the identity of the Airy function
$
\Ai(z)+\omega\Ai(\omega z)+\omega^2\Ai(\omega^2 z)=0,
$
the Airy parametrix defined in (\ref{A}) has the jump conditions:
\begin{subequations}\label{JoA}
  \begin{equation}\label{JoA1}
    A_+(z)=A_-(z)\left(\begin{array}{cc}1&0\\ \pm1&1\end{array}\right)
  \end{equation}
  for $z\in(0,\infty e^{\pm 2\pi/3})$, and
  \begin{equation}\label{JoA2}
    A_+(z)=A_-(z)\left(\begin{array}{cc}0&-1\\ 1&0\end{array}\right)
  \end{equation}
  for $z\in(-\infty,0)$, and
  \begin{equation}\label{JoA3}
    A_+(z)=A_-(z)\left(\begin{array}{cc}1&-1\\ 0&1\end{array}\right)
  \end{equation}
  for $z\in(0,\infty)$.
\end{subequations}
The Airy parametrix and its jump conditions are illustrated in Figure \ref{fig-A}.
\begin{figure}[htp]
\centering
\includegraphics{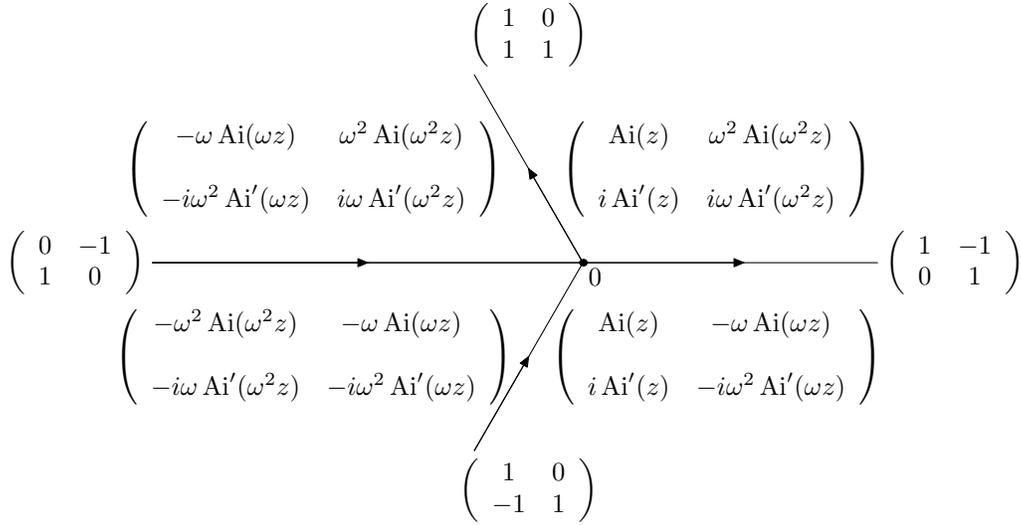}
\caption{The Airy parametrix and its jump
conditions.}\label{fig-A}
\end{figure}

\begin{rem}\label{rem-SoT}
  Now, we determine the precise shape of the curves $\Sigma_{T,\pm}^2$ and $\Sigma_{T,\pm}^6$ in Figure \ref{fig-SoT}.
  Recall the definition of the functions $F$ and $\wt F$ in (\ref{F}).
  On account of (\ref{prop-phi11}) and (\ref{prop-phi12}), we have
  \begin{eqnarray}\label{F-asymp}
    F(z)\sim\left(\frac{2n}{b\sqrt{b-a}}\right)^{2/3}(z-b)
  \end{eqnarray}
  as $z\to b$, and
  \begin{eqnarray}\label{F'-asymp}
    \wt F(z)\sim\left(\frac{2n}{a\sqrt{b-a}}\right)^{2/3}(a-z)
  \end{eqnarray}
  as $z\to a$.
  Furthermore, the function $F(z)$ is analytic in $U(b,\d_0)$
  and the function $\wt F(z)$ is analytic in $U(a,\d_0)$; see the choice of $\d_0$ in Remark \ref{rem-d}.
  We choose $\Sigma_{T,\pm}^6$ to be the inverse image of the rays $(0,\infty e^{\pm2\pi/3})$
  under the holomorphic map $F$,
  and $\Sigma_{T,\pm}^2$ to be the inverse image of the rays $(0,\infty e^{\mp2\pi/3})$
  under the holomorphic map $\wt F$.
\end{rem}

We recall the asymptotic expansions of the Airy function and its
derivative (cf. \cite[p. 392]{Ol97} or \cite[p. 47]{Wo89})
\begin{eqnarray}\label{Airy}
 \Ai(z)\sim \frac{z^{-1/4}}{2\sqrt\pi}e^{-\frac 2 3z^{3/2}}\sum_{s=0}^\infty
 \frac{(-1)^su_s}{(\frac 2 3z^{3/2})^s},\indent
 \Ai'(z)\sim -\frac{z^{1/4}}{2\sqrt\pi}e^{-\frac 2 3z^{3/2}}\sum_{s=0}^\infty
 \frac{(-1)^sv_s}{(\frac 2 3z^{3/2})^s}
\end{eqnarray}
as $z\to\infty$ with $|\arg z|<\pi$, where $u_s, v_s$ are constants with
$u_0=v_0=1$. Therefore, applying (\ref{Airy}) to (\ref{A}), we obtain
\begin{eqnarray}\label{A-asymp}
   {A}(z)=\frac {z^{-\s_3/4}}{2\sqrt\pi} \left(\begin{matrix}
    1&-i \\
    -i&1
  \end{matrix}\right)(I+O(|z|^{-3/2}))e^{-\frac 2 3z^{3/2}\s_3},\indent z\to\infty.
\end{eqnarray}
Define
\begin{eqnarray}\label{K}
  K(z):=n^{-\b\s_3/2}T(z)T_{par}^{-1}(z)n^{\b\s_3/2}.
\end{eqnarray}
The jump conditions of the function $K(z)$ are studied in the following proposition.
\begin{prop}\label{prop-JoK}
  Let $\Sigma_K$ be the contour shown in Figure \ref{fig-SoK}.
  The matrix-valued function $K(z)$ is analytic in $\bb{C}\setminus\Sigma_K$.
  On the contour $\Sigma_K$, the jump matrix $J_K(z):=K_-(z)^{-1}K_+(z)$ has the following explicit expressions.
  For $z\in\Sigma_{K,\pm}^1$, we have
  \begin{eqnarray}\label{JoK1}
  J_K(z)=n^{-\b\s_3/2}M\left(\begin{matrix}
    E/\wt E&0\\
    \\
    \cfrac{e^{2n\phi}}{\mp H}&\wt E/E
  \end{matrix}\right)M^{-1}n^{\b\s_3/2}.
  \end{eqnarray}
  For $z\in\Sigma_{K,\pm}^2$, we have
  \begin{eqnarray}\label{JoK2}
  J_K(z)=n^{-\b\s_3/2}M\left(\begin{matrix}
    1&\cfrac{\pm\wt H}{e^{2n\wt\phi}}\\
    \\
    0&1
  \end{matrix}\right)M^{-1}n^{\b\s_3/2}.
  \end{eqnarray}
  For $z\in\Sigma_{K,\pm}^3\cup\Sigma_{K,\pm}^5$, we have
  \begin{eqnarray}\label{JoK3/5}
  J_K(z)=n^{-\b\s_3/2}M\left(\begin{matrix}
    1&\cfrac{\pm\wt H \wt E}{e^{2n\wt\phi}E}\\
    \\
    \cfrac{e^{2n\phi}}{\mp H}&{\wt E}/{E}
  \end{matrix}\right)M^{-1}n^{\b\s_3/2}.
  \end{eqnarray}
  For $z\in\Sigma_{K,\pm}^6$, we have
  \begin{eqnarray}\label{JoK6}
  J_K(z)=n^{-\b\s_3/2}M\left(\begin{matrix}
     1&0\\
    \\
    \cfrac{e^{2n\phi}}{\pm H}&1
  \end{matrix}\right)M^{-1}n^{\b\s_3/2}.
  \end{eqnarray}
  For $z\in\Sigma_{K,\pm}^7$, we have
  \begin{eqnarray}\label{JoK7}
  J_K(z)=n^{-\b\s_3/2}M\left(\begin{matrix}
    1&\cfrac{\pm\wt H\wt E}{e^{2n\wt\phi}E}\\
    \\
    0&1
  \end{matrix}\right)M^{-1}n^{\b\s_3/2}.
  \end{eqnarray}
  For $z\in\Sigma_K^b$, we have
  \begin{eqnarray}\label{JoKb}
  J_K(z)=\sqrt\pi  n^{-\b\s_3/2}H^{\s_3/2}m F^{\s_3/4} {A}(F)e^{n\phi\sigma_3} H^{-\s_3/2}M^{-1}n^{\b\s_3/2}.
  \end{eqnarray}
  For $z\in\Sigma_K^a$, we have
  \begin{eqnarray}\label{JoKa}
  J_K(z)=\sqrt\pi n^{-\b\s_3/2}{\wt H}^{\s_3/2}\wt m\wt F^{-\s_3/4}\s_1
   {A}(\wt F)\s_1e^{n\wt\phi\sigma_3} {\wt H}^{-\s_3/2}M^{-1}n^{\b\s_3/2}.
  \end{eqnarray}
  For $z\in\Sigma_K^0$, we have
  \begin{eqnarray}\label{JoK0}
  J_K(z)=n^{-\b\s_3/2}MD^{\s_3}M^{-1}n^{\b\s_3/2}.
  \end{eqnarray}
  For $z\in\Sigma_{K,\pm}'$, we have
  \begin{eqnarray}\label{JoK'}
  J_K(z)=n^{-\b\s_3/2}M\left(\begin{matrix}
    1&0\\
    \\
    \cfrac{e^{2n\phi}}{\mp HD_+D_-}&1
  \end{matrix}\right)M^{-1}n^{\b\s_3/2}.
  \end{eqnarray}
Furthermore, the jump conditions of $K(z)$ on the positive real line are given as
\begin{subequations}\label{JoK-realline}
  \begin{equation}
    J_K(x)=n^{-\b\s_3/2}M\left(\begin{matrix}
    1&0\\
    \\
    \cfrac{e^{2n\wt\phi}}{\wt HD^2}&1
  \end{matrix}\right)M^{-1}n^{\b\s_3/2}
  \end{equation}
  for $0<x<\ve$, and
  \begin{equation}
    J_K(x)=n^{-\b\s_3/2}M\left(\begin{matrix}
    1&0\\
    \\
    \cfrac{e^{2n\wt\phi}}{\wt H}&1
  \end{matrix}\right)M^{-1}n^{\b\s_3/2}
  \end{equation}
  for $\ve<x<a-\ve$, and
  \begin{equation}
    J_K(x)=n^{-\b\s_3/2}M\left(\begin{matrix}
    1&\cfrac{- H}{e^{2n\phi}} \\
    \\
    0&1
  \end{matrix}\right)M^{-1}n^{\b\s_3/2}
  \end{equation}
  for $x>b+\ve$.
\end{subequations}
On the contour $\Sigma_K\setminus(\Sigma_K^a\cup\Sigma_K^b\cup\Sigma_K^0)$, the $L^\infty$ and $L^1$ norms of
the difference $J_K-I$ are exponentially small as $n\to\infty$.
On the contour $\Sigma_K^a\cup\Sigma_K^b\cup\Sigma_K^0$, we have
$$
\|J_K-I\|_{L^\infty(\Sigma_K^a\cup\Sigma_K^b\cup\Sigma_K^0)}=O(\frac 1 n),\indent n\to\infty.
$$

\begin{figure}[htp]
\centering
\includegraphics{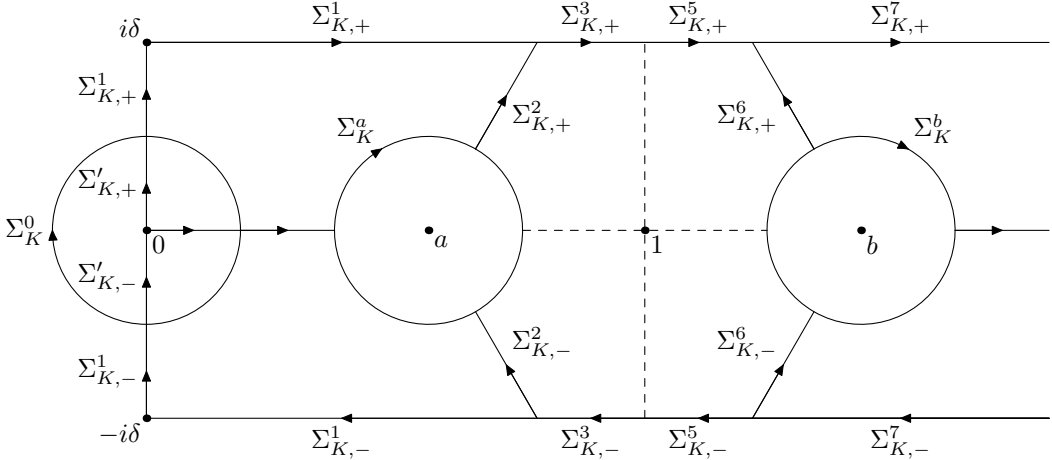}
\caption{The contour $\Sigma_K$.}\label{fig-SoK}
\end{figure}

\end{prop}

\begin{proof}
In Remark \ref{rem-SoT} we have shown that the function $F(z)$ is analytic in $U(b,\d_0)$
and the function $\wt F(z)$ is analytic in $U(a,\d_0)$. Since $0<\ve<\d<\d_0$,
we obtain from (\ref{m}) and (\ref{F-asymp}) that the matrix-valued function $mF^{\s_3/4}$ is analytic in $U(b,\ve)$,
and from (\ref{m'}) and (\ref{F'-asymp}) that the matrix-valued function $\wt m\wt F^{-\s_3/4}$ is analytic in $U(a,\ve)$.
Therefore, applying (\ref{JoA}) to (\ref{Tpb}) and (\ref{Tpa}) implies that the parametrix $T_{par}(z)$
possesses the same jump conditions as $T(z)$ in $U(a,\ve)\cup U(b,\ve)$;
see (\ref{JoT0}) and (\ref{JoT2/6}) in Proposition \ref{prop-JoT}.
  Thus, the function $K(z)$ defined in (\ref{K}) is analytic in $U(a,\ve)\cup U(b,\ve)$.
  Moreover, applying (\ref{JoT0}), (\ref{JoM}) and (\ref{Tp}) to (\ref{K})
  implies that the function $K(z)$ can be analytically continued to the interval $(a+\ve,b-\ve)$.
  Therefore, the analyticity of $K(z)$ in $\bb{C}\setminus\Sigma_K$
  is clear from the analyticity of $T(z)$ in $\bb{C}\setminus\Sigma_T$.

  Since the function $M(z)$ is analytic in $\bb{C}\setminus[a,b]$,
  we obtain (\ref{JoK1})-(\ref{JoK7}) from
  (\ref{JoT3})-(\ref{JoT7}), (\ref{JoT2/6}), (\ref{Tp}) and (\ref{K}).

  Since $T(z)$ has no jump on $\Sigma_K^a\cup\Sigma_K^b\cup\Sigma_K^0$,
  the formulas (\ref{JoKb})-(\ref{JoK0}) follow immediately from
  the definition of $T_{par}(z)$ in (\ref{Tp})-(\ref{Tpa}),
  and from the definition of $K(z)$ in (\ref{K}).

  For $z\in\Sigma_{K,\pm}'$, applying (\ref{Tp0}) to (\ref{K}) gives
  $$J_K(z)=n^{-\b\s_3/2}MD_-^{\s_3}T_-^{-1}T_+D_+^{-\s_3}M^{-1}n^{\b\s_3/2}.$$
  Thus, formula (\ref{JoK'}) follows from (\ref{JoT1}) and (\ref{JoM}).

  Moreover, a combination of (\ref{JoT0}), (\ref{Tp}), (\ref{Tp0}) and (\ref{K}) yields (\ref{JoK-realline}).

  From (\ref{G'-asymp}), (\ref{L-asymp}) and (\ref{M}), we obtain
  \begin{eqnarray}
  \label{M-asymp}
    |n^{-\b\s_3/2}M(z)n^{\b\s_3/2}|=O(1),\indent n\to\infty.
  \end{eqnarray}
  By applying (\ref{prop-phi13})-(\ref{prop-phi24}), (\ref{E/E'}), (\ref{H-asymp2}), (\ref{D}) and (\ref{M-asymp}) to
  (\ref{JoK1})-(\ref{JoK7}) and (\ref{JoK0})-(\ref{JoK-realline}), it follows
  that the norm $\|J_K-I\|_{L^\infty(\Sigma_K\setminus(\Sigma_K^a\cup\Sigma_K^b\cup\Sigma_K^0))}$ is exponentially small as $n\to\infty$.

  To prove the exponential decay property of the norm $\|J_K-I\|_{L^1(\Sigma_K\setminus(\Sigma_K^a\cup\Sigma_K^b\cup\Sigma_K^0))}$,
  we only need to show the $L^1$ norm of the difference $J_K-I$
  on the infinite contour $\Sigma_{K,\pm}^7\cup(b+\ve,\infty)$ is exponentially small as $n\to\infty$.
  Firstly, since $\phi''(x)>0$ for $x>b$ by (\ref{phi}) and the fact that $ab=1$, we have
  $$\phi(x)>\phi(b+\ve)+(x-b-\ve)\phi'(b+\ve)$$ for any $x>b+\ve$.
  Hence, we obtain
  $$\|e^{-2n\phi}\|_{L^1(b+\ve,\infty)}\le \frac{e^{-2n\phi(b+\ve)}}{2n\phi'(b+\ve)}.$$
  Applying (\ref{H-asymp2}) and (\ref{M-asymp}) to (\ref{JoK-realline}) implies that the norm
  $\|J_K-I\|_{L^1(b+\ve,\infty)}$ is exponentially small as $n\to\infty$.
  Furthermore, we observe from (\ref{phi}), (\ref{prop-phi22}) and (\ref{prop-phi23})
  that the $L^1$ norm of the function $e^{-2n\wt\phi}$ on the contour $\Sigma_{K,\pm}^7$ is also exponentially small as $n\to\infty$.
  Therefore, applying (\ref{E/E'}), (\ref{H-asymp2}) and (\ref{L-asymp}) to (\ref{JoK7}) implies that the norm
  $\|J_K-I\|_{L^1(\Sigma_{K,\pm}^7)}$ is exponentially small as $n\to\infty$.
  Thus, the exponential decay property of the norm $\|J_K-I\|_{L^1(\Sigma_K\setminus(\Sigma_K^a\cup\Sigma_K^b\cup\Sigma_K^0))}$ follows.

  Now, we prove the last statement of the proposition.
  For $z\in\Sigma_K^b$, applying (\ref{Mm}), (\ref{F}) and (\ref{A-asymp}) to (\ref{JoKb}) yields
  $$
  J_K(z)-I=n^{-\b\s_3/2}H(z)^{\s_3/2}m(z)O(\frac 1 n)m(z)^{-1}H(z)^{-\s_3/2}n^{\b\s_3/2},\indent n\to\infty.
  $$
  The estimate holds uniformly for $z\in\Sigma_K^b$. Thus, we obtain from (\ref{H-asymp1})
  $$
  \|J_K-I\|_{L^\infty(\Sigma_K^b)}=O(\frac 1 n),\indent n\to\infty.
  $$
  Similarly, a combination of (\ref{H-asymp1}), (\ref{Mm'}), (\ref{F}), (\ref{A-asymp}) and (\ref{JoKa}) gives
  $$
  \|J_K-I\|_{L^\infty(\Sigma_K^a)}=O(\frac 1 n),\indent n\to\infty.
  $$
  Finally, by (\ref{D}) we have $D(z)=1+O(1/n)$ uniformly for $z\in\Sigma_K^0$.
  Hence, it follows from (\ref{JoK0}) and (\ref{M-asymp}) that
  $$
  \|J_K-I\|_{L^\infty(\Sigma_K^0)}=O(\frac 1 n),\indent n\to\infty.
  $$
  This ends the proof of the proposition.
\end{proof}

\begin{prop}\label{prop-K}
The matrix-valued function $K(z)$ defined in (\ref{K}) is the unique solution to the Riemann-Hilbert problem:
\begin{enumerate}[(K1)]
\item  $K(z)$ is analytic in $\mathbb{C}\setminus \Sigma_K$;
\item  for $z\in\Sigma_K$, $K_+(z)=K_-(z)J_K(z)$, where $J_K(z)$ is given in Proposition \ref{prop-JoK};
\item for $z\in\mathbb{C}\setminus \Sigma_K$,
  $K(z)=I+O(|z|^{-1})$ as $z\to\infty$.
\end{enumerate}
Furthermore, as $n\to\infty$, we have $K(z)=I+O(1/n)$ uniformly for $z\in\bb{C}\setminus\Sigma_K$.
\end{prop}
\begin{proof}
  The analyticity condition (K1) and the jump conditions (K2) have been shown in Proposition \ref{prop-JoK}.
  The normalization condition (K3) is clear from
  the normalization conditions of the functions $T(z)$ and $M(z)$,
  and from the definition of the function $K(z)$.
  The uniqueness again follows from Liouville's theorem.
  Finally, as in \cite[Theorem 7.10]{DKMVZ99},
  we can obtain from Proposition \ref{prop-JoK} that
  $K(z)=I+O(1/n)$ as $n\to\infty$.
  The estimate is uniform for all $z\in\bb{C}\setminus\Sigma_K$.
\end{proof}

\section{Main Theorem}
We now state our main result of this paper.
\begin{thm}
For any $0<c<1$ and $1\le\b<2$,
let $\d_0>0$ be a sufficiently small number depending only on the constants $c$ and $\b$; see Remark \ref{rem-d}.
Recall from (\ref{v}) and (\ref{l}) that $v(z)=-z\log c$ and $l/2=\log\frac{b-a}4-1$,
where $a$ and $b$ are the Mhaskar-Rakhmanov-Saff numbers given in (\ref{ab}).
The functions $g, \phi, \wt\phi$ and $D$ are defined in (\ref{g}), (\ref{phi}), (\ref{phi'}) and (\ref{D}), respectively.
For any $0<\ve<\d<\d_0$, the large -- $n$ behavior of the monic Meixner polynomial $\pi_n(nz-\b/2)$
is given below (see Figure \ref{fig-rop}).

\begin{enumerate}[(i)]

\item
For $z\in\Omega^4\cup\Omega^\infty$, we have
\begin{eqnarray}\label{O4}
  \pi_n(nz-\b/2)=n^n e^{ng(z)}\frac {z^{(1-\b)/2}(\frac{\sqrt{z-a}+\sqrt{z-b}}2)^\b} {(z-a)^{1/4}(z-b)^{1/4}}
  \left[1+O(\frac 1 n)\right].
\end{eqnarray}

\item
For $z\in\Omega^1_\pm$, we have
\begin{eqnarray}\label{O1}
  \pi_n(nz-\b/2)&=& -2(-n)^n e^{nv(z)/2+nl/2}\frac {z^{(1-\b)/2}(\frac{\sqrt{b-z}+\sqrt{a-z}}2)^\b } {(a-z)^{1/4}(b-z)^{1/4}}
  \nonumber\\
  &&\times\bigg\{\sin(n\pi z-\b\pi/2)e^{-n\wt\phi(z)}\left[1+O(\frac 1 n)\right]
  \nonumber\\&&
  +O(n^\b e^{n\re\phi(z)})\bigg\}.
\end{eqnarray}

\item
For $z\in \Omega^0_l$, we have
\begin{eqnarray}\label{O0l}
  \pi_n(nz-\b/2)=D(z) n^n e^{ng(z)}\frac {(-z)^{(1-\b)/2}
  (\frac{\sqrt{b-z}+\sqrt{a-z}}2)^\b } {(b-z)^{1/4}(a-z)^{1/4}}
  \left[1+O(\frac 1 n)\right].
\end{eqnarray}

\item
For $z\in \Omega^0_{r,\pm}$, we have
\begin{eqnarray}\label{O0r}
  \pi_n(nz-\b/2)&=& - 2D(z)(-n)^n e^{nv(z)/2+nl/2}
  \frac {z^{(1-\b)/2}(\frac{\sqrt{b-z}+\sqrt{a-z}}2)^\b } {(a-z)^{1/4}(b-z)^{1/4}}
  \nonumber\\
  &&\times\bigg\{\sin(n\pi z-\b\pi/2)e^{-n\wt\phi(z)}\left[1+O(\frac 1 n)\right]
  \nonumber\\&&
  +O(n^\b e^{n\re\phi(z)})\bigg\}.
\end{eqnarray}

\item
Recall the definitions of the functions $F(z)$ and $\wt F(z)$ in (\ref{F}).
For $z\in \Omega^a$, we have
\begin{eqnarray}\label{Oa}
  \pi_n(nz-\b/2)&=&(-n)^n\sqrt\pi e^{nv(z)/2+nl/2}\nonumber\\
  &&\times
  \left\{
  \wt{\bf A}(z,n)
  \left[1+O(\frac 1 n)\right]
  +\wt{\bf B}(z,n)
  \left[1+O(\frac 1 n)\right]\right\},
\end{eqnarray}
where
\begin{eqnarray*}
  \wt{\bf A}(z,n)
&:=&\frac{(\frac{\sqrt{b-z}+\sqrt{a-z}}2)^\b+(\frac{\sqrt{b-z}-\sqrt{a-z}}2)^\b}{z^{(\b-1)/2}(b-z)^{1/4}(a-z)^{1/4}\wt F(z)^{-1/4}}
\\&&\times[\cos(n\pi z-\b\pi/2) \Ai(\wt F(z))-\sin(n\pi z-\b\pi/2)\Bi(\wt F(z))],
\end{eqnarray*}
and
\begin{eqnarray*}
  \wt{\bf B}(z,n)
&:=&\frac{(\frac{\sqrt{b-z}+\sqrt{a-z}}2)^\b-(\frac{\sqrt{b-z}-\sqrt{a-z}}2)^\b}{z^{(\b-1)/2}(b-z)^{1/4}(a-z)^{1/4}\wt F(z)^{1/4}}
\\&&\times[\cos(n\pi z-\b\pi/2) \Ai'(\wt F(z))-\sin(n\pi z-\b\pi/2)\Bi'(\wt F(z))].
\end{eqnarray*}

\item
For $z\in \Omega^b$, we have
\begin{eqnarray}\label{Ob}
  \pi_n(nz-\b/2)&=& n^n\sqrt\pi e^{nv(z)/2+nl/2}\nonumber\\
  &&\times
    \bigg\{
  {\bf A}(z,n)\left[1+O(\frac 1 n)\right]
  +{\bf B}(z,n)\left[1+O(\frac 1 n)\right]\bigg\},
\end{eqnarray}
where
\begin{eqnarray*}
  {\bf A}(z,n)
:=\frac{(\frac{\sqrt{z-a}+\sqrt{z-b}}2)^\b+(\frac{\sqrt{z-a}-\sqrt{z-b}}2)^\b}{z^{(\b-1)/2}(z-a)^{1/4}(z-b)^{1/4} F(z)^{-1/4}}\Ai(F(z)),
\end{eqnarray*}
and
\begin{eqnarray*}
  {\bf B}(z,n)
:=-\frac{(\frac{\sqrt{z-a}+\sqrt{z-b}}2)^\b-(\frac{\sqrt{z-a}-\sqrt{z-b}}2)^\b}{z^{(\b-1)/2}(z-a)^{1/4}(z-b)^{1/4} F(z)^{1/4}}\Ai'(F(z)).
\end{eqnarray*}

\item
Let $z=\frac{b-a}2 \cos u+\frac{b+a}2=-\frac{b-a}2 \cos\wt u+\frac{b+a}2$. We have
\begin{eqnarray}\label{O2}
  \pi_n(nz-\b/2)&=& 2(-n)^n e^{nv(z)/2+nl/2}\frac {z^{(1-\b)/2}(\frac{b-a}4)^{\b/2}} {(z-a)^{1/4}(b-z)^{1/4}}
  \nonumber\\
  &&\times\bigg\{\cos(n\pi z-\b\pi/2+\pi/4+\b\wt u/2\mp in\wt\phi(z)) \left[1+O(\frac 1 n)\right]
    \nonumber\\&&+O(n^{-1}e^{n|\re\wt\phi(z)|+n\pi|\im z|})\bigg\}
\end{eqnarray}
for $z\in \Omega^2_\pm$, and
\begin{eqnarray}\label{O3}
  \pi_n(nz-\b/2)&=& 2n^n e^{nv(z)/2+nl/2}\frac {z^{(1-\b)/2}(\frac{b-a}4)^{\b/2}} {(z-a)^{1/4}(b-z)^{1/4}}
  \nonumber\\
  &&\times\bigg\{\cos(\pi/4-\b u/2\mp in\phi(z))\left[1+O(\frac 1 n)\right]
  \nonumber\\&&
  +O(n^{-1}e^{n|\re\phi(z)|})\bigg\}
\end{eqnarray}
for $z\in \Omega^3_\pm$.
In view of (\ref{phi'}) and the fact that $\wt u+u=\pi$,
the asymptotic formulas (\ref{O2}) and (\ref{O3}) are exactly the same.
\end{enumerate}

\begin{figure}[htp]
\centering
\includegraphics{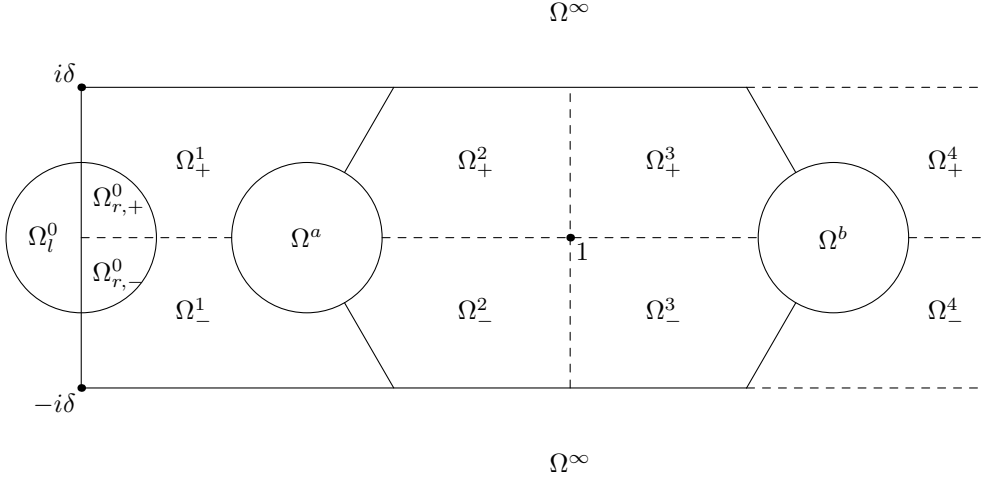}
\caption{Regions of asymptotic approximations.
A dashed line indicates that the
\newline \leftline{\hspace{1.7cm}
 asymptotic formulas on its two sides are the same.}
 }\label{fig-rop}
\end{figure}

\end{thm}

\begin{proof}
  Applying (\ref{Q}), (\ref{W}) and (\ref{E/E'}) to (\ref{R}), we obtain
  \begin{subequations}\label{thm-R}
    \begin{equation}\label{thm-R1}
      U(z)=R(z)\left[\prod\limits_{j=0}^{n-1}(z-X_j)\right]^{\s_3}\left(\begin{matrix}
    1&0\\
    \\
    \cfrac{\mp Ee^{nv}}{W\wt E}&1
  \end{matrix}\right)
    \end{equation}
    for $\re z\in(0,1)$ and $\im z\in(0,\pm\d)$, and
    \begin{equation}\label{thm-R2}
      U(z)=R(z)\left[\prod\limits_{j=0}^{n-1}(z-X_j)\right]^{\s_3}\left(\begin{matrix}
    1&\cfrac{\mp W\wt Ee^{-nv}}{Ee^{\mp 2i\pi (nz-\b/2)}}\\
    \\
    0&1
  \end{matrix}\right)
    \end{equation}
    for $\re z\in(1,\infty)$ and $\im z\in(0,\pm\d)$, and
    \begin{equation}\label{thm-R3}
      U(z)=R(z)\left[\prod\limits_{j=0}^{n-1}(z-X_j)\right]^{\s_3}
    \end{equation}
    for $\re z\notin[0,\infty)$ or $\im z\notin[-\d,\d]$.
  \end{subequations}
  It is easily seen from (\ref{E}) and (\ref{G}) that
  \begin{eqnarray}\label{thm-1}
  \prod\limits_{j=0}^{n-1}(z-X_j)=\left(\frac z{z-1}\right)^{\frac{1-\b}2}E(z)e^{ng(z)-G(z)}.
  \end{eqnarray}
  For the sake of convenience, we put
  \begin{eqnarray}\label{U'}
    \wt U(z):=e^{(-nl/2)\s_3}U(z)e^{(-nv(z)/2)\s_3}.
  \end{eqnarray}
  Thus, we have from (\ref{P}) and (\ref{U}) that
  \begin{eqnarray}\label{p2U'}
    \wt U_{11}(z)=n^{-n}e^{-nv(z)/2-nl/2}\pi_n(nz-\b/2).
  \end{eqnarray}
  A combination of (\ref{gphi}), (\ref{H}), (\ref{S}) and (\ref{thm-R})-(\ref{U'}) yields
  \begin{subequations}\label{thm-S}
    \begin{equation}\label{thm-S1}
      \wt U(z)=S(z)E^{\s_3}e^{-n\phi\s_3}
    \left(\begin{matrix}
    1&0\\
    \\
    \cfrac{\mp E}{H\wt E}&1
  \end{matrix}\right)\left(\cfrac z{z-1}\right)^{\frac{1-\b}2\s_3}
    \end{equation}
    for $\re z\in(0,1)$ and $\im z\in(0,\pm\d)$, and
    \begin{equation}\label{thm-S2}
      \wt U(z)=S(z)E^{\s_3}e^{-n\phi\s_3}
  \left(\begin{matrix}
    1&\cfrac{\mp H\wt E}{Ee^{\mp 2i\pi (nz-\b/2)}}\\
    \\
    0&1
  \end{matrix}\right)\left(\cfrac z{z-1}\right)^{\frac{1-\b}2\s_3}
    \end{equation}
    for $\re z\in(1,\infty)$ and $\im z\in(0,\pm\d)$, and
    \begin{equation}\label{thm-S3}
      \wt U(z)=S(z)E^{\s_3}e^{-n\phi\s_3}\left(\cfrac z{z-1}\right)^{\frac{1-\b}2\s_3}
    \end{equation}
    for $\re z\notin[0,\infty)$ or $\im z\notin[-\d,\d]$.
  \end{subequations}

  For $z\in\Omega^4\cup\Omega^\infty$, we apply (\ref{T}), (\ref{Tp}) and (\ref{K}) to (\ref{thm-S}), and obtain
  \begin{eqnarray}\label{thm-41}
  e^{-L\s_3/2}\wt Ue^{L\s_3/2}&=&(e^{-L\s_3/2}n^{\b\s_3/2}Kn^{-\b\s_3/2}e^{L\s_3/2})(e^{-L\s_3/2}Me^{L\s_3/2})
  \nonumber\\
  &&\times
  e^{-n\phi\s_3}\left(\begin{matrix}
    1&*\\
    0&1
  \end{matrix}\right)\left(\frac z{z-1}\right)^{\frac{1-\b}2\s_3};
  \end{eqnarray}
  here and below, we denote by $*$ some irrelevant quantity which does not effect our final result.
  From (\ref{L-asymp}) and Proposition \ref{prop-K}, we see
  \begin{eqnarray}\label{thm-KL}
  e^{-L\s_3/2}n^{\b\s_3/2}Kn^{-\b\s_3/2}e^{L\s_3/2}=I+O(\frac 1 n),\indent n\to\infty.
  \end{eqnarray}
  Therefore, applying (\ref{G'-asymp}) and (\ref{M}) to (\ref{thm-41}), we have
  $$
  \wt U_{11}=e^{-n\phi}\cfrac{z^{\frac{1-\b}2}(\frac{\sqrt{z-a}+\sqrt{z-b}}2)^\b}{(z-a)^{1/4}(z-b)^{1/4}}\left[1+O(\frac 1 n)\right].
  $$
  Thus, formula (\ref{O4}) follows from (\ref{gphi}) and (\ref{p2U'}).

  For $z\in\Omega_\pm^1$, we apply (\ref{T}), (\ref{Tp}) and (\ref{K}) to (\ref{thm-S}). The result is
  \begin{eqnarray}\label{thm-11}
  e^{-L\s_3/2}\wt Ue^{L\s_3/2}&=&
  (e^{-L\s_3/2}n^{\b\s_3/2}Kn^{-\b\s_3/2}e^{L\s_3/2})
  (e^{-L\s_3/2}Me^{L\s_3/2})e^{-n\phi\s_3}\nonumber\\
  &&\times\left(\begin{matrix}
    E/\wt E&0\\
    \\
    \mp e^L/H&\wt E/E
  \end{matrix}\right)\left(\frac z{z-1}\right)^{\frac{1-\b}2\s_3}.
  \end{eqnarray}
  Recall from (\ref{W}) and (\ref{H}) that $H(z)=[z/(z-1)]^{1-\b}W(z)$ and
  $$W(z)=\frac{2ni\pi \G(nz+\b/2)c^{-\b/2}}{\G(nz+1-\b/2)}.$$
  We observe by Stirling's formula that,
  the function $H(z)^{-1}$ is uniformly bounded for $\re z\ge0$, as $n\to\infty$.
  From (\ref{L-asymp}), it follows that
  $|e^L/H(z)|=O(n^\b)$.
  Therefore, applying (\ref{G'-asymp}), (\ref{M}) and (\ref{thm-KL}) to (\ref{thm-11}) gives
  \begin{eqnarray}\label{thm-12}
  \wt U_{11}=\cfrac{z^{\frac{1-\b}2}(\frac{\sqrt{b-z}+\sqrt{a-z}}2)^\b}{(b-z)^{1/4}(a-z)^{1/4}}\left\{
  (E/\wt E)e^{-n\phi\mp i\pi(1-\b)/2}\left[1+O(\frac 1 n)\right]+O(n^\b e^{n\re\phi})\right\}.
  \end{eqnarray}
  Since
  $$(E/\wt E)e^{-n\phi\mp i\pi(1-\b)/2}=-2(-1)^ne^{-n\wt\phi}\sin(n\pi z-\b\pi/2)$$
  by (\ref{phi'}) and (\ref{E/E'}),
  the asymptotic formula (\ref{O1}) follows from (\ref{p2U'}) and (\ref{thm-12}).

  For $z\in\Omega^0$, the proof of (\ref{O0l}) and (\ref{O0r}) is similar to
  that of (\ref{O4}) and (\ref{O1}).
  The only difference comes from the definition of the parametrix $T_{par}(z)$ in (\ref{Tp}) and (\ref{Tp0}).
  We thus replace $M$ by $MD^{\s_3}$ in (\ref{thm-41}) and (\ref{thm-11});
  consequently, the asymptotic formulas (\ref{O0l}) and (\ref{O0r}) are simply the formulas (\ref{O4}) and (\ref{O1})
  multiplied by the function $D(z)$.

  For $z\in\Omega^a$, we first consider the case $\arg \wt F(z)\in(\mp2\pi/3,\mp\pi)$.
  In view of (\ref{F'-asymp}), this region is approximately the same as the region $\arg (z-a)\in(0,\pm\pi/3)$.
  Hence, we obtain from (\ref{T}) and Remark \ref{rem-SoT} that
  $$T(z)=S(z)\left(\begin{matrix}
     \wt E &{\mp\wt H \wt E}/{e^{2n\wt\phi}}\\
     \\
    0&1/\wt E
  \end{matrix}\right).$$
  Applying this and (\ref{K}) to (\ref{thm-S}) gives
  \begin{eqnarray}\label{thm-a1}
  \wt H^{-\s_3/2}\wt U\wt H^{\s_3/2}
  &=&
  (\wt H^{-\s_3/2}n^{\b\s_3/2}Kn^{-\b\s_3/2}\wt H^{\s_3/2})
  (\wt H^{-\s_3/2}T_{par}\wt H^{\s_3/2}e^{-n\wt\phi\s_3})
  \nonumber\\
  &&\times e^{n\wt\phi\s_3}\wt H^{-\s_3/2}
  \left(\begin{matrix}
     1/\wt E &\cfrac{\pm\wt H \wt E}{e^{2n\wt\phi}}\\
     \\
    0&\wt E
  \end{matrix}\right)e^{-n\phi\s_3}\left(\begin{matrix}
    E&0\\
    \\
    \cfrac{\mp1}{H\wt E}&1/E
  \end{matrix}\right)\wt H^{\s_3/2}
  \nonumber\\  &&\times
  \left(\frac z{z-1}\right)^{\frac{1-\b}2\s_3}.
  \end{eqnarray}
  Coupling (\ref{A}) and (\ref{Tpa}) yields
  \begin{eqnarray}\label{thm-a21}
  &&\wt H^{-\s_3/2}T_{par}\wt H^{\s_3/2}e^{-n\wt\phi\s_3}
  \nonumber\\
  &=&\sqrt\pi\wt m \wt F^{-\s_3/4}\left(\begin{array}{cc}
  -i\omega^2\Ai'(\omega\wt F)
  &-i\omega \Ai'(\omega^2\wt F)\\
\\
 -\omega\Ai(\omega\wt F)&
   -\omega^2\Ai(\omega^2\wt F)
  \end{array}\right)
  \nonumber\\
  &=&\sqrt\pi\wt m \wt F^{-\s_3/4}\left(\begin{array}{cc}
  \Ai'(\wt F)&-\Bi'(\wt F)\\
\\
  -i\Ai(\wt F)& i\Bi(\wt F)
  \end{array}\right)\left(\begin{array}{cc}
  i/2&i/2
\\
  -1/2&1/2
  \end{array}\right)
  \end{eqnarray}
  for $\arg \wt F(z)\in(-2\pi/3,-\pi)$, and
  \begin{eqnarray}\label{thm-a22}
  &&\wt H^{-\s_3/2}T_{par}\wt H^{\s_3/2}e^{-n\wt\phi\s_3}
  \nonumber\\
  &=&\sqrt\pi\wt m \wt F^{-\s_3/4}\left(\begin{array}{cc}
  i\omega \Ai'(\omega^2\wt F)
  &-i\omega^2\Ai'(\omega\wt F)\\
\\
  \omega^2\Ai(\omega^2\wt F)
  &-\omega\Ai(\omega\wt F)
  \end{array}\right)
  \nonumber\\
  &=&\sqrt\pi\wt m \wt F^{-\s_3/4}\left(\begin{array}{cc}
  \Ai'(\wt F)&-\Bi'(\wt F)\\
\\
  -i\Ai(\wt F)&i\Bi(\wt F)
  \end{array}\right)\left(\begin{array}{cc}
  -i/2&i/2
\\
  -1/2&-1/2
  \end{array}\right)
  \end{eqnarray}
  for $\arg \wt F(z)\in(2\pi/3,\pi)$.
  Here, we have made use of the identities
\begin{eqnarray}\label{Bi}
  2\omega\Ai(\omega z)=-\Ai(z)+i\Bi(z),\indent 2\omega^2\Ai(\omega^2 z)=-\Ai(z)-i\Bi(z).
\end{eqnarray}
  A combination of (\ref{phi'}), (\ref{E/E'}) and (\ref{H/H'}) implies
  \begin{eqnarray}\label{thm-a3}
  &&e^{n\wt\phi\s_3}\wt H^{-\s_3/2}\left(\begin{matrix}
     1/\wt E &\cfrac{\pm\wt H \wt E}{e^{2n\wt\phi}}\\
     \\
    0&\wt E
  \end{matrix}\right)e^{-n\phi\s_3}\left(\begin{matrix}
    E&0\\
    \\
    \cfrac{\mp1}{\wt EH}&1/E
  \end{matrix}\right)\wt H^{\s_3/2}\left(\frac z{z-1}\right)^{\frac{1-\b}2\s_3}
  \nonumber\\
  &=&(-1)^n\left(\begin{matrix}
    \mp ie^{\mp i\pi(nz-\b/2)}\ &ie^{\pm i\pi(nz-\b/2)}\wt E/E\\
    \\
    -ie^{\pm i\pi(nz-\b/2)}&\pm ie^{\pm i\pi(nz-\b/2)}\wt E/E
  \end{matrix}\right)\left(\frac z{1-z}\right)^{\frac{1-\b}2 \s_3}
  \nonumber\\
  &=&(-1)^n\left(\begin{array}{cc}
  \mp i&-1
\\
  -i&\pm 1
  \end{array}\right)\left(\begin{matrix}
    \cos(n\pi z-\b\pi/2)\ &*\\
    \\
    \sin(n\pi z-\b\pi/2)\ &*
  \end{matrix}\right)\left(\frac z{1-z}\right)^{\frac{1-\b}2 \s_3},
  \end{eqnarray}
  where the $*$ stands for some irrelevant quantities.
  Applying (\ref{thm-a21})-(\ref{thm-a3}) to (\ref{thm-a1}) gives
  \begin{eqnarray}\label{thm-a4}
  \wt H^{-\s_3/2}\wt U\wt H^{\s_3/2}
  &=&
  (\wt H^{-\s_3/2}n^{\b\s_3/2}Kn^{-\b\s_3/2}\wt H^{\s_3/2})
  (\sqrt\pi\wt m \wt F^{-\s_3/4})(-1)^n
  \nonumber\\
  &&\times
  \left(\begin{matrix}
    \cos(n\pi z-\b\pi/2) \Ai'(\wt F)-\sin(n\pi z-\b\pi/2)\Bi'(\wt F)\ &*\\
    \\
    -i\cos(n\pi z-\b\pi/2) \Ai(\wt F)+i\sin(n\pi z-\b\pi/2)\Bi(\wt F)\ &*
  \end{matrix}\right)
  \nonumber\\  &&\times
  \left(\frac z{1-z}\right)^{\frac{1-\b}2 \s_3}.%
  \end{eqnarray}
  From (\ref{H-asymp1}) and Proposition \ref{prop-K}, we have
  \begin{eqnarray}
    \label{thm-KH'}
    \wt H^{-\s_3/2}n^{\b\s_3/2}Kn^{-\b\s_3/2}\wt H^{\s_3/2}=I+O(1/n),\indent n\to\infty.
  \end{eqnarray}
  Coupling (\ref{thm-a4}) and (\ref{thm-KH'}) yields
  \begin{eqnarray*}
  \wt U_{11}(z)&=&(-1)^n\sqrt\pi\left(\frac z{1-z}\right)^{\frac{1-\b}2}\\
  &&\times \bigg\{-i\wt m_{12}\wt F^{1/4}
  [\cos(n\pi z-\b\pi/2) \Ai(\wt F)-\sin(n\pi z-\b\pi/2)\Bi(\wt F)]\left[1+O(\frac 1 n)\right]\\&&+
  \wt m_{11}\wt F^{-1/4}
  [\cos(n\pi z-\b\pi/2) \Ai'(\wt F)-\sin(n\pi z-\b\pi/2)\Bi'(\wt F)]\left[1+O(\frac 1 n)\right]\bigg\}.
  \end{eqnarray*}
  Thus, formula (\ref{Oa}) follows from (\ref{m'}) and (\ref{p2U'}).

  Now, we consider the case $\arg \wt F(z)\in(0,\mp2\pi/3)$.
  In view of Remark \ref{rem-SoT}, we obtain from (\ref{T}) that $T(z)=S(z)\wt E(z)^{\s_3}$.
  Applying this and (\ref{K}) to (\ref{thm-S}) gives
  \begin{eqnarray}\label{thm-a1'}
  \wt H^{-\s_3/2}\wt U\wt H^{\s_3/2}
  &=&
  (\wt H^{-\s_3/2}n^{\b\s_3/2}Kn^{-\b\s_3/2}\wt H^{\s_3/2})
  (\wt H^{-\s_3/2}T_{par}\wt H^{\s_3/2}e^{-n\wt\phi\s_3})
  \nonumber\\
  &&\times e^{n\wt\phi\s_3}\wt H^{-\s_3/2} e^{-n\phi\s_3}\left(\begin{matrix}
    E/\wt E&0\\
    \\
    \mp1/H&\wt E/E
  \end{matrix}\right)\wt H^{\s_3/2}\left(\frac z{z-1}\right)^{\frac{1-\b}2\s_3}.\qquad\qquad
  \end{eqnarray}
  A combination of (\ref{Tpa}), (\ref{A}) and (\ref{Bi}) yields
  \begin{eqnarray}\label{thm-a2'}
  \wt H^{-\s_3/2}T_{par}\wt H^{\s_3/2}e^{-n\wt\phi\s_3}
  =\sqrt\pi\wt m \wt F^{-\s_3/4}\left(\begin{array}{cc}
  \Ai'(\wt F)&-\Bi'(\wt F)\\
\\
  -i\Ai(\wt F)& i\Bi(\wt F)
  \end{array}\right)\left(\begin{array}{cc}
  \pm i/2&i
\\
  -1/2&0
  \end{array}\right).
  \end{eqnarray}
  Coupling (\ref{phi'}), (\ref{E/E'}) and (\ref{H/H'}) implies
  \begin{eqnarray}\label{thm-a3'}
  &&e^{n\wt\phi\s_3}\wt H^{-\s_3/2}e^{-n\phi\s_3}\left(\begin{matrix}
    E/\wt E&0\\
    \\
    \mp1/H&\wt E/E
  \end{matrix}\right)\wt H^{\s_3/2}\left(\frac z{z-1}\right)^{\frac{1-\b}2\s_3}\nonumber\\
  &=&(-1)^n\left(\begin{matrix}
    - 2\sin(n\pi z-\b\pi/2)&0\\
    \\
    -ie^{\pm i\pi(nz-\b/2)}\ &\pm ie^{\pm i\pi(nz-\b/2)}\wt E/E
  \end{matrix}\right)\left(\frac z{1-z}\right)^{\frac{1-\b}2 \s_3}
  \nonumber\\
  &=&(-1)^n\left(\begin{array}{cc}
  0&-2
\\
  -i&\pm 1
  \end{array}\right)\left(\begin{matrix}
    \cos(n\pi z-\b\pi/2)\ &*\\
    \\
    \sin(n\pi z-\b\pi/2)&0
  \end{matrix}\right)\left(\frac z{1-z}\right)^{\frac{1-\b}2 \s_3},
  \end{eqnarray}
  where the $*$ stands for some irrelevant quantity.
  Applying (\ref{thm-a2'}) and (\ref{thm-a3'}) to (\ref{thm-a1'}), we again obtain (\ref{thm-a4}).
  A combination of (\ref{m'}), (\ref{p2U'}), (\ref{thm-a4}) and (\ref{thm-KH'}) yields (\ref{Oa}).

  For $z\in\Omega^b$, we only consider the case $\arg F(z)\in(\pm2\pi/3,\pm\pi)$ here.
  The case $\arg F(z)\in(0,\pm2\pi/3)$ is much simpler and we omit the details.
  On account of Remark \ref{rem-SoT}, we obtain from (\ref{T}) that
  $$T(z)=S(z)\left(\begin{matrix}
     E &0\\
     \\
    {e^{2n\phi}}/({\pm HE})&1/  E
  \end{matrix}\right).$$
  Applying this and (\ref{K}) to (\ref{thm-S}) gives
  \begin{eqnarray}\label{thm-b1}
   H^{-\s_3/2}\wt U H^{\s_3/2}
   &=&
  ( H^{-\s_3/2}n^{\b\s_3/2}Kn^{-\b\s_3/2} H^{\s_3/2})
  ( H^{-\s_3/2}T_{par} H^{\s_3/2}e^{-n\phi\s_3})
  \nonumber\\
  &&\times e^{n\phi\s_3} H^{-\s_3/2}\left(\begin{matrix}
     1/ E &0\\
     \\
    \cfrac{e^{2n\phi}}{\mp HE}& E
  \end{matrix}\right)e^{-n\phi\s_3}\left(\begin{matrix}
    E&\cfrac{\mp H\wt E}{e^{\mp 2i\pi (nz-\b/2)}}\\
    \\
    0&1/E
  \end{matrix}\right) H^{\s_3/2}\nonumber\\
  &&\times\left(\frac z{z-1}\right)^{\frac{1-\b}2\s_3}\nonumber\\
  &=&
  ( H^{-\s_3/2}n^{\b\s_3/2}Kn^{-\b\s_3/2} H^{\s_3/2})
  ( H^{-\s_3/2}T_{par} H^{\s_3/2}e^{-n\phi\s_3})\nonumber\\
  &&\times\left(\begin{matrix}
    1&*\\
    \mp 1&*
  \end{matrix}\right)
  \left(\frac z{z-1}\right)^{\frac{1-\b}2\s_3},
  \end{eqnarray}
  where the $*$ again stands for some irrelevant quantities.
  Coupling (\ref{Tpb}) and (\ref{A}) yields
  \begin{eqnarray}\label{thm-b21}
   &&
   H^{-\s_3/2}T_{par} H^{\s_3/2}e^{-n\phi\s_3}
   \nonumber\\
  &=&\sqrt\pi m  F^{\s_3/4}\left(\begin{array}{cc}
  -\omega\Ai(\omega F)&
  \omega^2\Ai(\omega^2 F)\\
\\
  -i\omega^2\Ai'(\omega F)
  &i\omega \Ai'(\omega^2 F)
  \end{array}\right)
  \nonumber\\
  &=&\sqrt\pi m  F^{\s_3/4}\left(\begin{array}{cc}
  \Ai( F)& \Bi( F)\\
\\
  i\Ai'( F)&i\Bi'( F)
  \end{array}\right)\left(\begin{array}{cc}
  1/2&-1/2
\\
  -i/2&-i/2
  \end{array}\right)
  \end{eqnarray}
  for $\arg F(z)\in(2\pi/3,\pi)$, and
  \begin{eqnarray}
  \label{thm-b22}
   &&H^{-\s_3/2}T_{par} H^{\s_3/2}e^{-n\phi\s_3}\nonumber\\
  &=&\sqrt\pi m  F^{\s_3/4}\left(\begin{array}{cc}
  -\omega^2\Ai(\omega^2 F)
  &-\omega\Ai(\omega F)\\
\\
  -i\omega \Ai'(\omega^2 F)
  &-i\omega^2\Ai'(\omega F)
  \end{array}\right)
  \nonumber\\
  &=&\sqrt\pi m  F^{\s_3/4}\left(\begin{array}{cc}
  \Ai( F)& \Bi( F)\\
\\
  i\Ai'( F)&i\Bi'( F)
  \end{array}\right)\left(\begin{array}{cc}
  1/2&1/2
\\
  i/2&-i/2
  \end{array}\right)
  \end{eqnarray}
  for $\arg F(z)\in(-2\pi/3,-\pi)$. Here, we have made use of (\ref{Bi}).
  Moreover, from (\ref{H-asymp1}) and Proposition \ref{prop-K}, we obtain
  \begin{eqnarray}
    \label{thm-KH}
    H^{-\s_3/2}n^{\b\s_3/2}Kn^{-\b\s_3/2} H^{\s_3/2}=I+O(1/n),\indent n\to\infty.
  \end{eqnarray}
  A combination of (\ref{thm-b1})-(\ref{thm-KH}) gives
  \begin{eqnarray*}
  \wt U_{11}(z)=\sqrt\pi\left(\frac z{z-1}\right)^{\frac{1-\b}2}\bigg\{m_{11}F^{1/4}\Ai(F)\left[1+O(\frac 1 n)\right]+i
  m_{12}F^{-1/4}\Ai'(F)\left[1+O(\frac 1 n)\right]\bigg\}.
  \end{eqnarray*}
  Thus, formula (\ref{Ob}) follows from (\ref{m}) and (\ref{p2U'}).

  For $z\in\Omega_\pm^2$,
  similar to the proof of (\ref{Oa}) in the case $\arg\wt F(z)\in(\mp2\pi/3,\mp\pi)$,
  we obtain (\ref{thm-a1}) from (\ref{T}), (\ref{K}) and (\ref{thm-S}).
  Also, equality (\ref{thm-a3}) follows from a combination of (\ref{phi'}), (\ref{E/E'}) and (\ref{H/H'}).
  Setting $z=-\frac{b-a}2\cos\wt u+\frac{b+a}2$; we have $\sqrt{b-z}\pm i\sqrt{z-a}=\sqrt{b-a}e^{\pm i\wt  u/2}$.
  Since $\wt H(z)=(1-z)^{\b-1}e^{-V(z)}$ by (\ref{HV}),
  and $|\wt G(z)|+|V(z)+L|=O(1/n)$ by (\ref{V-asymp}), (\ref{G'-asymp}) and (\ref{L-asymp}),
  we obtain from (\ref{M}) and (\ref{Tp}) that
  \begin{eqnarray}\label{thm-21}
  \wt H^{-\s_3/2}T_{par}\wt H^{\s_3/2}e^{-n\wt\phi\s_3}
  &=&\cfrac {(1-z)^{\frac{1-\b}2\s_3}(\frac{b-a}4)^{\b/2}} {(b-z)^{1/4}(z-a)^{1/4}}
  \nonumber\\&&\times\left[
  \left(\begin{matrix}
    e^{\mp i\b\wt u/2\pm i \pi/4}\quad{}&i e^{\pm i\b\wt u/2\pm i \pi/4}\\
    \\
    -i e^{\pm i\b\wt u/2\pm i \pi/4}\quad{}&e^{\mp i\b\wt u/2\pm i \pi/4}
  \end{matrix}\right)+O(\frac 1 n)\right]
  \nonumber\\&&\times e^{-n\wt\phi\s_3}.
  \end{eqnarray}
  Since
  \begin{eqnarray*}
  &&\left(\begin{matrix}
    e^{\mp i\b\wt u/2\pm i \pi/4}\quad{}&i e^{\pm i\b\wt u/2\pm i \pi/4}\\
    \\
    -i e^{\pm i\b\wt u/2\pm i \pi/4}\quad{}&e^{\mp i\b\wt u/2\pm i \pi/4}
  \end{matrix}\right)e^{-n\wt\phi\s_3}
  \\&=&
  \left(\begin{array}{cc}
  2\cos(\pi/4+\b\wt u/2\mp in\wt\phi)\ &-2\sin(\pi/4+\b\wt u/2\mp in\wt\phi)\\
\\
  O(e^{n|\re\wt\phi|})&O(e^{n|\re\wt\phi|})
  \end{array}\right)\left(\begin{array}{cc}
  \pm i/2&i/2
\\
  -1/2&\pm1/2
  \end{array}\right),
  \end{eqnarray*}
  it follows from (\ref{thm-a1}), (\ref{thm-a3}) and (\ref{thm-21}) that
  \begin{eqnarray}\label{thm-22}
  \wt H^{-\s_3/2}\wt U\wt H^{\s_3/2}
  &=&
  \bigg(\wt H^{-\s_3/2}n^{\b\s_3/2}Kn^{-\b\s_3/2}\wt H^{\s_3/2}\bigg)
  \frac {(1-z)^{\frac{1-\b}2\s_3}(\frac{b-a}4)^{\b/2}} {(b-z)^{1/4}(z-a)^{1/4}}
  \nonumber\\&&\times
  \left(\begin{array}{cc}
  \wt r_{11}&*\\
  \\
  O(e^{n|\re\wt\phi|+n\pi|\im z|})\ &*
  \end{array}\right)
  \left(\frac z{1-z}\right)^{\frac{1-\b}2 \s_3},
  \end{eqnarray}
  where the $*$ stands for some irrelevant quantities, and
  $$\wt r_{11}=\cos(n\pi z-\b\pi/2+\pi/4+\b\wt u/2\mp in\wt\phi)\left[1+O(\frac 1 n)\right]
    +O(n^{-1}e^{n|\re\wt\phi|+n\pi|\im z|}).$$
  For $z\in\Omega_\pm^2$, we have from (\ref{HV}) and (\ref{V-asymp}) that
  $$|n^{-\b}\wt H(z)(1-z)^{1-\b}|+|(1-z)^{\b-1}\wt H(z)^{-1}n^\b|=O(1)$$
  as $n\to\infty$.
  In view of $K(z)=I+O(1/n)$ by Proposition \ref{prop-K}, we obtain from (\ref{thm-22})
  $$\wt U_{11}(z)=\frac{2(-1)^nz^{\frac{1-\b}2}(\frac{b-a}4)^{\b/2}\wt r_{11}} {(b-z)^{1/4}(z-a)^{1/4}}.$$
  Coupling this with (\ref{p2U'}) yields (\ref{O2}).

  For $z\in\Omega_\pm^3$,
  similar to the proof of (\ref{Ob}) in the case $\arg F(z)\in(\pm2\pi/3,\pm\pi)$,
  we obtain (\ref{thm-b1}) from (\ref{T}), (\ref{K}) and (\ref{thm-S}).
  Set $z=\frac{b-a}2\cos u+\frac{b+a}2$, and we have $\sqrt{z-a}\pm i\sqrt{b-z}=\sqrt{b-a}e^{\pm iu/2}$.
  Since $H(z)=(z-1)^{\b-1}e^{-V(z)}$ by (\ref{HV}),
  and $|\wt G(z)|+|V(z)+L|=O(1/n)$ by (\ref{V-asymp}), (\ref{G'-asymp}) and (\ref{L-asymp}),
  we obtain from (\ref{M}) and (\ref{Tp}) that
  \begin{eqnarray}\label{thm-31}
  H^{-\s_3/2}T_{par} H^{\s_3/2}e^{-n\phi\s_3}
  &=&\frac {(z-1)^{\frac{1-\b}2\s_3}(\frac{b-a}4)^{\b/2}} {(b-z)^{1/4}(z-a)^{1/4}}
  \nonumber\\&&\times \left[\left(\begin{matrix}
    e^{\pm i\b u/2\mp i \pi/4}\quad{}&-i e^{\mp i\b u/2\mp i \pi/4}\\
    \\
    i e^{\mp i\b u/2\mp i \pi/4}\quad{}&e^{\pm i\b u/2\mp i \pi/4}
  \end{matrix}\right)+O(\frac 1 n)\right]
  \nonumber\\&&\times e^{-n\phi\s_3}.
  \end{eqnarray}
  Since
  \begin{eqnarray*}
  &&\left(\begin{matrix}
    e^{\pm i\b u/2\mp i \pi/4}\quad{}&-i e^{\mp i\b u/2\mp i \pi/4}\\
    \\
    i e^{\mp i\b u/2\mp i \pi/4}\quad{}&e^{\pm i\b u/2\mp i \pi/4}
  \end{matrix}\right)
  e^{-n\phi\s_3}
  \\&=&
  \left(\begin{array}{cc}
  2\cos(\pi/4-\b u/2\mp in\phi)\quad{}&-ie^{\mp i\b u\mp i\pi/4+n\phi}\\
\\
  O(e^{n|\re\phi|})\quad{}&O(e^{n|\re\phi|})
  \end{array}\right)
  \left(\begin{array}{cc}
  1\quad{}&0
\\ \\
  \pm1\quad{}&1
  \end{array}\right),
  \end{eqnarray*}
  applying (\ref{thm-31}) to (\ref{thm-b1}) gives
  \begin{eqnarray}\label{thm-32}
  H^{-\s_3/2}\wt U H^{\s_3/2}
   &=&
  \bigg( H^{-\s_3/2}n^{\b\s_3/2}Kn^{-\b\s_3/2} H^{\s_3/2}\bigg)
  \frac {(z-1)^{\frac{1-\b}2\s_3}(\frac{b-a}4)^{\b/2}} {(b-z)^{1/4}(z-a)^{1/4}}
  \nonumber\\&&\times
  \left(\begin{array}{cc}
  r_{11}&*\\
  \\
  O(e^{n|\re\phi|})\ &*
  \end{array}\right)
  \left(\frac z{z-1}\right)^{\frac{1-\b}2 \s_3},
  \end{eqnarray}
  where the $*$ stands for some irrelevant quantities, and
  $$r_{11}=\cos(\pi/4-\b u/2\mp in\phi)\left[1+O(\frac 1 n)\right]
    +O(n^{-1}e^{n|\re\phi|}).$$
  For $z\in\Omega_\pm^3$, we have from (\ref{HV}) and (\ref{V-asymp}) that
  $$|n^{-\b}H(z)(z-1)^{1-\b}|+|(z-1)^{\b-1}H(z)^{-1}n^\b|=O(1)$$
  as $n\to\infty$.
  In view of $K(z)=I+O(1/n)$ by Proposition \ref{prop-K}, we obtain from (\ref{thm-32})
  $$\wt U_{11}(z)=\frac{2z^{\frac{1-\b}2}(\frac{b-a}4)^{\b/2}r_{11}} {(b-z)^{1/4}(z-a)^{1/4}}.$$
  Coupling this with (\ref{p2U'}) yields (\ref{O3}).
  Moreover, since $z=\frac{b-a}2 \cos u+\frac{b+a}2=-\frac{b-a}2 \cos\wt u+\frac{b+a}2$, we have $\wt u+u=\pi$.
  In view of (\ref{phi'}), the two asymptotic formulas (\ref{O2}) and (\ref{O3}) are exactly the same.
\end{proof}

\section{Numerical Evidence}

In this section we provide some numerical computations by using our results in Theorem 1.
Choosing $c=0.5$, it is easily seen from (\ref{ab}) that $a\approx0.17157$ and $b\approx5.82843$.
We also fix $\b=1.5$.
Since the polynomial degree $n$ should be reasonably large, we set $n=100$.
The approximate values of $\pi_n(nz-\b/2)$ are obtained by using the asymptotic formulas given in Theorem 1.
We use formula (\ref{O4}) for $z=-1$ and $z=100$, formula (\ref{O0l})-(\ref{O0r}) for $z=\pm0.001$,
formula (\ref{O1}) for $z=0.05$, formula (\ref{Oa}) for $z=0.171$ and $z=0.172$,
formula (\ref{O2}) or (\ref{O3}) for $z=2$, and formula (\ref{Ob}) for $z=5.828$ and $z=5.829$.
The true values of $\pi_n(nz-\b/2)$ can be obtained from (\ref{Meixner}) and (\ref{monic Meixner}).
The numerical results are presented in Table 1.

\begin{table}[h]
\centering
\begin{tabular}{llrrr}
\hline
&& \multicolumn{1}{c}{True value}  & \multicolumn{1}{c}{Approximate value}\\
\hline
$z=-1$ 
&&
$1.99529\times 10^{233}$ & $1.99473\times 10^{233}$\\
$z=-0.001$ 
&&
$8.36624\times 10^{187}$ & $8.35137\times 10^{187}$\\
$z=0.001$ 
&&
$3.07930\times 10^{187}$ & $3.07272\times 10^{187}$\\
$z=0.05$ 
&&
$-2.51701\times 10^{180}$ & $-2.51507\times 10^{180}$\\
$z=0.171$ 
&&
$-9.12697\times 10^{174}$ & $-9.12530\times 10^{174}$\\
$z=0.172$ 
&&
$-1.22035\times 10^{175}$ & $-1.22003\times 10^{175}$\\
$z=2$ 
&&
$-4.71541\times 10^{201}$ & $-4.70772\times 10^{201}$\\
$z=5.828$ 
&&
$2.78146\times 10^{259}$ & $2.78231\times 10^{259}$\\
$z=5.829$ 
&&
$2.86933\times 10^{259}$ & $2.87018\times 10^{259}$\\
$z=100$ 
&&
$2.16586\times 10^{399}$ & $2.16586\times 10^{399}$\\
\end{tabular}
\caption{ The true values and approximate values of $\pi_n(nz-\b/2)$
for $c=0.5$,
\newline \leftline{\hspace{1.7cm}
$\b=1.5$ and $n=100$. Note that $a\approx0.17157$ and $b\approx5.82843$.}
}
\end{table}

\section{Comparison with Earlier Results}

In this section, we compare our formulas in Theorem 1 with those given in \cite{JW98} and \cite{JW99}.
First, we introduce two notations. Let
\begin{eqnarray}\label{Jin-az}
\a:=z-\b/(2n)
\end{eqnarray}
and
\begin{eqnarray}\label{Jin-mpi}
m_n(n\a;\b,c):=(1-1/c)^n\pi_n(nz-\b/2).
\end{eqnarray}
Two different asymptotic formulas for $m_n(n\a;\b,c)$
are given in \cite[(6.9) and (6.27)]{JW98}; both in terms of parabolic cylinder functions.
To study the large and small zeros of the Meixner polynomials,
these two formulas are transformed to (2.35) and (4.19) in \cite{JW99}.
Here, we intend to show the equivalence between our equation (\ref{Ob}) and equation (2.35) in \cite{JW99},
and also the equivalence between our equation (\ref{O1}) and equation (4.19) in \cite{JW99}.

In view of \cite[(2.34)]{JW99}, we rewrite the formula \cite[(2.35)]{JW99} as follows:
\begin{eqnarray}\label{Jin10}
  m_n(n\a;\b,c)\sim (-1)^n\sqrt{2\pi}n^{n+1/6}e^{n(\g+\eta^2/4-1/2)}
  c^{-1/6}(1+\sqrt c)^{2/3-\b}\Ai(n^{2/3}(\eta-2)),
\end{eqnarray}
where $\g$ is a constant and $\eta$ is a function of $\a$.
The constant $\g$ and the function $\eta$
could be solved from the following two equations (cf. \cite[(3.12)-(3.13)]{JW98}):
\begin{eqnarray}
  \a\log(1-w_+/c)-\a\log(1-w_+)-\log(-w_+)=-\log u_-+\eta u_--u_-^2/2+\g,\label{Jin11}\\
  \a\log(1-w_-/c)-\a\log(1-w_-)-\log(-w_-)=-\log u_++\eta u_+-u_+^2/2+\g.\label{Jin12}
\end{eqnarray}
The saddle points $w_\pm$ and $u_\pm$ are as given below (cf. \cite[(2.5)]{JW98} and \cite[(3.8)]{JW98}):
$$
w_\pm=\frac{1+c+\a c-\a\pm\sqrt{(1+c+\a c-\a)^2-4c}}2,\indent u_\pm=\eta/2\pm\sqrt{\eta^2/4-1};
$$
see also \cite[(2.4)-(2.5)]{JW99}.
Adding (\ref{Jin11}) to (\ref{Jin12}) gives
\begin{eqnarray}\label{Jin13}
  \eta^2/4+\g+1/2&=&-\frac{\a+1}2\log c.
\end{eqnarray}
Subtracting (\ref{Jin11}) from (\ref{Jin12}) yields
\begin{eqnarray}\label{Jin14}
  &&(\eta/2)\sqrt{\eta^2/4-1}+\log(\eta/2-\sqrt{\eta^2/4-1})\nonumber\\
  &=&
  \frac\a2\log\frac{(1-w_-/c)(1-w_+)}{(1-w_+/c)(1-w_-)}+\frac1 2\log\frac{w_+}{w_-}.
\end{eqnarray}
From the definition of $\phi$-function in (\ref{phi}), we have $\phi(b)=0$ and
$$
\phi'(\a)=\log\frac{\a(b+a)-2+2\sqrt{(\a-a)(\a-b)}}{\a(b-a)}=
\frac1 2\log\frac{(1-w_-/c)(1-w_+)}{(1-w_+/c)(1-w_-)}.
$$
Therefore, we obtain from (\ref{Jin14})
$$
  \phi(\a)=(\eta/2)\sqrt{\eta^2/4-1}+\log(\eta/2-\sqrt{\eta^2/4-1})\sim\frac 2 3(\eta-2)^{3/2},
$$
where we have made use of the restriction $\eta-2=O(n^{-2/3})$; see \cite[p.284]{JW99}.
In view of (\ref{F}), we then have
\begin{eqnarray}\label{Jin15}
  F(\a)\sim n^{2/3}(\eta-2).
\end{eqnarray}
A combination of (\ref{Jin-az})-(\ref{Jin10}), (\ref{Jin13}) and (\ref{Jin15}) gives
\begin{eqnarray}\label{Jin18}
  \pi_n(nz-\b/2)\sim\sqrt{2\pi}n^{n+1/6}e^{-n}
  c^{-nz/2+n/2+\b/4-1/6}(1-c)^{-n}(1+\sqrt c)^{2/3-\b}\Ai(F(z)).
\end{eqnarray}
Applying (\ref{prop-phi11}) to (\ref{F}) implies
$$
\frac{F(z)}{z-b}\sim\left(\frac{2n}{b\sqrt{b-a}}\right)^{2/3}=n^{2/3}c^{-1/6}(1-c)(1+\sqrt c)^{-4/3}.
$$
Therefore, we have
\begin{eqnarray}\label{Jin19}
  \frac{(\frac{\sqrt{z-a}+\sqrt{z-b}}2)^\b+(\frac{\sqrt{z-a}-\sqrt{z-b}}2)^\b}
  {z^{(\b-1)/2}(z-a)^{1/4}(z-b)^{1/4} F(z)^{-1/4}}
  \sim\sqrt 2n^{1/6}c^{\b/4-1/6}(1+\sqrt c)^{2/3-\b}.
\end{eqnarray}
Moreover, it is easy to see that
\begin{eqnarray}
  \frac{(\frac{\sqrt{z-a}+\sqrt{z-b}}2)^\b-(\frac{\sqrt{z-a}-\sqrt{z-b}}2)^\b}
  {z^{(\b-1)/2}(z-a)^{1/4}(z-b)^{1/4} F^{1/4}}=O(n^{-1/6}).
\end{eqnarray}
From (\ref{v}) and (\ref{l}), we obtain
\begin{eqnarray}\label{Jin-vl}
  e^{nv/2+nl/2}=e^{-n}c^{-nz/2+n/2}(1-c)^{-n}.
\end{eqnarray}
Hence, we can derive (\ref{Jin18}) again by applying (\ref{Jin19})-(\ref{Jin-vl}) to (\ref{Ob}).
This establishes the equivalence between (\ref{Ob}) and \cite[(2.35)]{JW99}.

Applying \cite[(3.4)]{JW99} and \cite[(3.11)-(3.12)]{JW99} to \cite[(4.19)]{JW99}, we have
\begin{eqnarray}\label{Jin20}
  m_n(n\a;\b,c)&=&\frac{2n^{n\a}n!}{\G(n\a+1)}\a^{n\a+1/2}e^{n(\g+\eta^2/4-\a/2)}
  e^{n[-\a\log(-\eta/2+\sqrt{\eta^2/4-\a})-(\eta/2)\sqrt{\eta^2/4-\a}]}\nonumber\\
  &&\times\frac{-h(u_-)}{(\eta^2-4\a)^{1/4}\sqrt{-u_-}}[1+O(1/n)]\{\sin n\pi\a+O(\a^{-1/2}e^{-2\ve_0n})\}.
\end{eqnarray}
Here again, $\g$ is a constant and $\eta$ is a function of $\a$, and
they can be solved from the two equations (cf. \cite[(3.23)-(3.24)]{JW98}):
\begin{eqnarray}
  \a\log(1-w_+/c)-\a\log(1-w_+)-\log w_+=-\a\log u_++\eta u_+-u_+^2/2+\g,\label{Jin21}\\
  \a\log(1-w_-/c)-\a\log(1-w_-)-\log w_-=-\a\log u_-+\eta u_--u_-^2/2+\g.\label{Jin22}
\end{eqnarray}
The saddle points $w_\pm$ and $u_\pm$ are given by (cf. \cite[(2.5)]{JW98} and \cite[(3.22)]{JW98})
$$
w_\pm=\frac{1+c+\a c-\a\pm\sqrt{(1+c+\a c-\a)^2-4c}}2,\indent u_\pm=\eta/2\pm\sqrt{\eta^2/4-\a};
$$
see also \cite[(3.3)-(3.4)]{JW99}.
Adding (\ref{Jin21}) to (\ref{Jin22}) yields
\begin{eqnarray}\label{Jin23}
  \eta^2/4+\g=-\frac{\a+1}2\log c-\a/2+\frac{\a}2\log \a.
\end{eqnarray}
Subtracting (\ref{Jin21}) from (\ref{Jin22}) gives
\begin{eqnarray}\label{Jin24}
  (-\eta/2)\sqrt{\eta^2/4-\a}-\a\log(-\eta/2+\sqrt{\eta^2/4-\a})+(\a/2)\log\a\nonumber\\
  =
  \frac\a2\log\frac{(w_-/c-1)(1-w_+)}{(w_+/c-1)(1-w_-)}+\frac 1 2\log\frac{w_+}{w_-}.
\end{eqnarray}
Recall the definition of $\wt\phi$-function in (\ref{phi'}). We have $\wt\phi(a)=0$, and
$$
\wt\phi'(\a)=\log\frac{-\a(b+a)+2+2\sqrt{(a-\a)(b-\a)}}{\a(b-a)}
=-\frac1 2\log\frac{(w_-/c-1)(1-w_+)}{(w_+/c-1)(1-w_-)}.
$$
Therefore, we obtain from (\ref{Jin24})
\begin{eqnarray}\label{Jin25}
  (-\eta/2)\sqrt{\eta^2/4-\a}-\a\log(-\eta/2+\sqrt{\eta^2/4-\a})+(\a/2)\log\a
  =-\wt\phi(\a).
\end{eqnarray}
Furthermore, a direct calculation shows that
\begin{eqnarray}\label{Jin26}
  \frac{-h(u_-)}{(\eta^2-4\a)^{1/4}\sqrt{-u_-}}
  =-\sqrt\a[(a-\a)(b-\a)]^{-1/4}(1-w_-)^{-\b}.
\end{eqnarray}
Using the equality
$$
\frac{-\a(b+a)+2+2\sqrt{(a-\a)(b-\a)}}{b-a}=c^{-1/2}(1-w_-)^2(\frac{\sqrt{b-\a}+\sqrt{a-\a}}2)^2
$$
and Taylor's expansion, we have
\begin{eqnarray}\label{Jin27}
  e^{-n\wt\phi(\a)+n\wt\phi(z)}&=&e^{(\b/2)\wt\phi'(\a)+O(1/n)}\nonumber\\
  &=&\left[\frac{-\a(b+a)+2+2\sqrt{(a-\a)(b-\a)}}{\a(b-a)}\right]^{\b/2}[1+O(1/n)]\nonumber\\
  &=&\a^{-\b/2}c^{-\b/4}(1-w_-)^\b\left(\frac{\sqrt{b-\a}+\sqrt{a-\a}}2\right)^\b[1+O(1/n)].
\end{eqnarray}
It can be shown by Stirling's formula that
\begin{eqnarray}\label{Jin28}
  \frac{2n^{n\a}n!}{\G(n\a+1)}\a^{n\a+1/2}e^{-n\a/2}
  =2n^ne^{-n+n\a/2}[1+O(1/n)].
\end{eqnarray}
Applying (\ref{Jin23}) and (\ref{Jin25})-(\ref{Jin28}) to (\ref{Jin20}) gives
\begin{eqnarray}\label{Jin29}
  m_n(n\a;\b,c)&=&-2n^ne^{-n}c^{-n\a/2-\b/4-n/2}e^{-n\wt\phi(z)}
  \frac{\a^{(1-\b)/2}(\frac{\sqrt{b-\a}+\sqrt{a-\a}}2)^\b}{[(a-\a)(b-\a)]^{1/4}}\bigg[1+O(1/n)\bigg]\nonumber\\
  &&\times\{\sin n\pi\a+O(\a^{-1/2}e^{-2\ve_0n})\},
\end{eqnarray}
which is exactly the same as (\ref{O1}) in view of (\ref{Jin-az}), (\ref{Jin-mpi}) and (\ref{Jin-vl}).

\end{document}